\documentclass[12pt]{amsart}

\usepackage{amssymb,epsfig,amsfonts,setspace,dsfont}
 \usepackage{amsmath,amscd,amsthm,array,geometry}

\newtheorem{thm}{Theorem}[section]

\newtheorem{lem}[thm]{Lemma}
\newtheorem{prop}[thm]{Proposition}
\newtheorem{conj}[thm]{Conjecture}
\newtheorem{cor}[thm]{Corollary}

\theoremstyle{definition}
\newtheorem{defn}[thm]{Definition}
\newtheorem{rem}[thm]{Remark}
\newtheorem{exam}[thm]{Example}

\newcommand{\G}{\mathbf{G}}
\newcommand{\Galois}{\Gamma_{E/F}}
\newcommand{\localsystem}{\mathcal{L}_{\widetilde{\rho}}}
\newcommand{\vol}{\mathrm{vol}}

\newcommand{\loc}{\mathcal{L}}
\newcommand{\M}{\mathcal{M}}
\newcommand{\tr}{\mathrm{tr}}
\newcommand{\Hom}{\mathrm{Hom}}

\newcommand{\g}{\mathfrak{g}}

\newcommand{\overpi}{\widetilde{\pi}}
\newcommand{\K}{\mathcal{K}}
\newcommand{\U}{\mathcal{U}}
\newcommand{\adele}{\mathbb{A}}
\newcommand{\finadele}{\mathbb{A}^{\mathrm{fin}}}

\newcommand{\W}{\mathcal{W}}
\newcommand{\R}{\mathbb{R}}
\newcommand{\C}{\mathbb{C}}
\newcommand{\PGL}{\mathbf{PGL}}
\newcommand{\GL}{\mathbf{GL}}
\newcommand{\MS}{\mathrm{MS}}

\begin{document}

    \title{Equivariant Torsion and Base Change}
    \author{Michael Lipnowski}
\address{Mathematics Department \\
Duke University, Box 90320 \\
Durham, NC 27708-0320, USA}
\email{malipnow@math.duke.edu}
\maketitle

\begin{abstract}
What is the true order of growth of torsion in the cohomology of an arithmetic group?  Let $D$ be a quaternion over an imaginary quadratic field $F.$  Let $E/F$ be a cyclic Galois extension with $\Galois = \langle \sigma \rangle.$  We prove lower bounds for ``the Lefschetz number of $\sigma$ acting on torsion cohomology" of certain Galois-stable arithmetic subgroups of $D_E^\times.$  For these same subgroups, we unconditionally prove a would-be-numerical consequence of the existence of a hypothetical base change map for torsion cohomology.
\end{abstract}

\setcounter{section}{-1}
\tableofcontents

\section{Introduction}

Let $\G / \mathbb{Q}$ be a semisimple group.  Let $K \subset G = \G(\R)$ be a maximal compact subgroup and $X = X_{\G} = G / K$ the symmetric space of maximal compact subgroups of $G.$ Fix an arithmetic subgroup $\Gamma \subset \G(\mathbb{Q}).$  Let $\rho: \G \rightarrow \GL(V)$ be a homomorphism of algebraic groups over $\mathbb{Q}$ and let $M \subset V$ be a $\Gamma$-stable lattice.

Suppose that $\G / \mathbb{Q}$ is anisotropic and that $M_{\R}$ is strongly acyclic \cite[$\S 4$]{BV}.  Let $\Gamma_n \subset \Gamma$ be a sequence of subgroups for which the injectivity radius of $\Gamma_n \backslash X$ approaches infinity.  Bergeron and Venkatesh \cite{BV} prove that if the fundamental rank \footnote{fundamental rank of $G/K =  \mathrm{rank}(G) - \mathrm{rank}(K),$ where $\mathrm{rank}$ denotes the dimension of any maximal torus, not necessarily split} of $X$ equals 1, then
\begin{equation*}
\liminf_n  \frac{\sum_i \log |H^i(\Gamma_n,M)_{\mathrm{tors}}|}{[\Gamma : \Gamma_n]} > 0. 
\end{equation*}     
Little is known about the true order of growth of $\log |H^{*}(\Gamma_n,M)|$ for $X$ of fundamental rank $\neq 1.$   


Let $F$ be an imaginary quadratic field and let $E/F$ a Galois extension cyclic of degree $p.$  Let $D$ be a quaternion algebra over $F$ and $\G$ be the adjoint group of the unit group of $F.$  One goal of this paper is to prove lower bounds for the amount of torsion in the cohomology of locally symmetric spaces for $X_{\PGL_1(D_E)},$ which has fundamental rank $p > 1.$   

Calegari and Venkatesh have proposed an analogue of Langlands functoriality for torsion cohomology \cite[$\S 2$]{CV}.  The hypothetical existence of torsion base change functoriality leads one to predict that torsion cohomology on locally symmetric spaces for $X_{\PGL_1(D)}$ - proven to be abundant in \cite{BV} - can be lifted to torsion cohomology on locally symmetric spaces for $X_{\PGL_1(D_E)}.$  A second goal of this paper is to unconditionally prove a numerical relationship between the cohomology of certain locally symmetric spaces for $X_{\PGL_1(D_E)}$ and ``matching" locally symmetric spaces for $X_{\PGL_1(D)}$ which is consistent with base change for torsion.

\subsection{Notational setup for statement of main results}\label{notn}
\subsubsection{Analytic torsion and Reidemeister torsion}
Let $\loc \rightarrow \M$ be a local system of $\C$ vector spaces equipped with a Hermitian metric.  Let $\sigma,$ of prime order $p,$ act equivariantly by isometries on $\loc \rightarrow \M.$  Let $\Delta_j$ denote the $j$-form Laplace operator for $\loc.$  Let 
\begin{equation*}
\tau_{\sigma}(\M, \loc) := \prod_{j = 0}^{\dim \M} \mathrm{det}'(\sigma^{*} \circ \Delta_{j,\loc})^{j (-1)^j}
\end{equation*}
denote the $\sigma$-equivariant analytic torsion of $\loc \rightarrow \M$ \cite[
$\S 1.1$]{Lip1}. Here, $\mathrm{det}'(A)$ denotes the regularized product of non-zero eigenvalues of $A$ \cite[$\S 2$]{BV}.  The untwisted analytic torsion is defined to be
\begin{equation*}
\tau(\M,\loc) := \tau_1(\M,\loc).
\end{equation*}
We let $RT_{\sigma}(\M, \loc)$ denote the equivariant Reidemeiester torsion of $\loc \rightarrow \M$   \cite[
$\S 1.3$]{Lip1}.  We let $RT(\M,\loc) := RT_1(\M,\loc)$ denote untwisted Reidemeister torsion.

\subsubsection{The locally symmetric space}
Let $D$ be a quaternion algebra over an imaginary quadratic field $F.$ 
  Let $E/F$ be a cyclic Galois extension of prime degree $p$, with Galois group $\Galois$;
  we fix a generator $\sigma$ for $\Galois$. 
  Let $\G$ denote the adjoint group of $\underline{D}^{\times}$, considered as an $F$-algebraic group.  We form the associated  locally symmetric space
 $$\M_\U = \G(E) \backslash \G(\mathbb{A}_E) / \mathcal{K} \; \mathcal{U}.$$
where $\mathcal{U}$ is a compact open Galois-stable subgroup of $\G(\mathbb{A}_E^{\mathrm{fin}})$ and
 $\mathcal{K}$ is a Galois stable maximal compact subgroup of $\G(E_{\mathbb{R}}).$  There is a unique - up to scaling - $\G(E_\R)$ invariant metric on $\G(E_\R) / \K$ which descends to a metric on $\M_\U.$  The group $\Galois$ acts on $\mathcal{M}_{\mathcal{U}}$ isometrically with respect to any such metric.  For a discussion of normalization of this metric, see $\S \ref{homogeneous}.$

\subsubsection{The local system}
Let $N/F$ be a finite extension and $V$ an $N$-vector space.  Let $\widetilde{\rho}:  \mathrm{Res}_{E/F} \G \rtimes \Galois \rightarrow \mathrm{Res}_{N/F} \GL(V)$ be an algebraic representation (over $F$).  We fix an ``integral structure'' within $V$, i.e. an $O_N$-lattice inside $V$; let $\mathcal{U}_0$ be its stabilizer inside $\G(\mathbb{A}_E^{\mathrm{fin}})$.

 The representation $\widetilde{\rho}$ gives rise to a local system of $N$-vector spaces $V_{\widetilde{\rho}}\rightarrow \M_\U$ with an action of $\Galois$.
 If $\mathcal{U} \subset \mathcal{U}_0,$ the integral structure on $V$ yields an integral structure on this local system, i.e. a local system of $O_N$-modules, which we denote
 $$ \localsystem \rightarrow \mathcal{M}_{\mathcal{U}}.$$
This integral structure is discussed in greater detail in $\S \ref{rational}.$

\subsection{Statement of main results}
\label{statementsandmethods}


We briefly describe the main theorems of this paper.  Let
$$M_U = \G(F) \backslash \G(\mathbb{A}_F) / KU, \text{ where } K = \mathcal{K} \cap \mathbf{G}(F_{\mathbb{R}}), U = \mathcal{U} \cap \mathbf{G}(\mathbb{A}_F^{\mathrm{fin}}).$$  
%

\newtheorem*{sample}{Sample Theorem (Comparison of analytic torsion)}

\begin{sample} 
Let $E/F$ be an everywhere unramified Galois extension of odd prime degree.  Suppose that $\mathcal{U}$ is a parahoric level structure (see Definition \ref{parahoric}) and $\rho$ and $\widetilde{\rho}$ match (see $\S \ref{compatibleloc}$).  Then for any complex embedding $\iota: N \hookrightarrow \C,$
\begin{equation*}
\tau_{\sigma}(\mathcal{ M}_{\mathcal{U}},  \loc_{\widetilde{\rho},\iota}) = \tau(M_U, L_{\rho, \iota})^p.
\end{equation*}
A more general statement, which allows any $E/F$ which is everywhere tamely ramified, is proven in $\S \ref{abstractmatching}, \S \ref{tamematching}.$  The relationship in the more general case between $\tau_{\sigma}(\mathcal{ M}_{\mathcal{U}},  \localsystem)$ and $\tau(M_U, L_{\rho})$ has the same flavor but is not as simply stated. 
\end{sample}  

Spectral comparisons such as the sample theorem, in conjunction with Cheeger-M\"{u}ller theorems (see $\S \ref{maintools}$), have consequences for torsion in the cohomology of $\localsystem.$  In order to describe these implications, we use the following notational shorthand:
\begin{itemize}
\item
$\sum {}^{*}$ denotes alternating sum.

\item
$P$ denotes the $p$-cyclotomic polynomial $P(x) = x^{p-1} + x^{p-2} + ... + 1.$

\item
For any $\mathbb{Z}[\sigma]$-module $A$ and any polynomial $h \in \mathbb{Z}[x],$ we let $A^{h(\sigma)} = \{ a \in A: h(\sigma)a = 0 \}.$
\end{itemize}

\newtheorem*{sizeoftorsion}{Sample Theorem (relationship between sizes of torsion subgroups)}

\begin{sizeoftorsion}
Let $E/F$ be an everywhere tamely ramified Galois extension of odd prime degree with $\Galois = \langle \sigma \rangle.$  Let the places where $E/F$ is ramified and the places where $D$ is ramified be disjoint.  Suppose that

\begin{itemize}
\item
$\rho, \widetilde{\rho}$ are matching representations of the sort described in $\S \ref{weirdloc}.$

\item
The level structure $\U$  is \emph{tamely parahoric} at each unramified place of $E/F$ (see Definition \ref{tameparahoric}).
\end{itemize}

Then there is an explicit finite collection of compact open subgroups $U \subset \G(\finadele_F)$ and explicit constants $c_U$ such that
$$\sum {}^{*} \log |H^i(\M_\U, \localsystem)_{\mathrm{tors}}^{\sigma - 1}| - \frac{1}{p-1} \log |H^i(\M_\U, L_{\rho})^{P(\sigma)}| = \sum {}^{*} \log |H^i(M_U, L_{\rho})_{\mathrm{tors}}|  + n \log p.$$
for some integer $n.$  Furthermore, $n$ can be bounded linearly by $\dim_{\mathbb{F}_p} H^i(M_U, L_{\rho, \mathbb{F}_p}).$
\end{sizeoftorsion}

For a more precise statement, see Theorem \ref{cohomologycomparison}.  An appropriate generalization of the sample theorem also has consequences for growth of torsion in the cohomology of the spaces $\M_\U.$

%
%

\newtheorem*{lotsoftorsion}{Corollary (Growth of torsion for fundamental rank $>1$)}

\begin{lotsoftorsion} 
Let $E/F$ be everywhere tamely ramified with $[E:F]$ odd.  Let the places where $E/F$ is ramified and the places where $D$ is ramified be disjoint. \medskip
 
Let $\mathcal{U}_N \subset \mathcal{U}_0$ denote a sequence of compact open subgroups of $\G(\mathbb{A}_E^{\mathrm{fin}})$ such that

\begin{itemize}
\item
The injectivity radius of $\mathcal{M}_{\mathcal{U}_N}$ approaches $\infty.$

\item
The level structures $\mathcal{U}_N$  are \emph{tamely parahoric} at each unramified place of $E/F$ (see Definition \ref{tameparahoric}).

\item
The local systems $L_{\rho} \rightarrow M_{U_N}$ form a strongly acyclic family (see $\S \ref{stronglyacyclic}$), where $\rho$ and $\widetilde{\rho}$ are matching representations (see $\S \ref{compatibleloc}$).

\item
$L_{\rho}$ has ``not too much mod $p$ cohomology".  More precisely,  

$$\frac{\log  | H^{*}(M_{U_N}, L_{\rho,\mathbb{F}_p })|} { \vol( M_{U_N})} \rightarrow 0.  $$

\end{itemize} 

Then it follows that

$$ \limsup_N \frac{\log  | H^{*}(\mathcal{M}_{\mathcal{U}_N}, \localsystem)| }{ \vol( \mathcal{M}_{\mathcal{U}_N})^{ \frac{1}{p} }  } > 0.$$

\end{lotsoftorsion}

In fact, we prove a more refined result (see $\S \ref{cohomologygrowth}$) which is analogous to the statement that for many $\ell$ dividing $|H^{*}(\mathcal{M}_{\mathcal{U}_N}, \localsystem)|,$ ``the Lefschetz numbers of $\sigma$ acting on the cohomology of $\mathcal{L}_{\widetilde{\rho}, \mathbb{F}_{\ell}}$ are big."

%
%
%

\begin{rem} 
The fourth assumption of the corollary, namely $L_{\rho}$ having ``not too much mod $p$ cohomology," cannot be removed at the present time.  However, it is expected always to be true.  A conjecture to this effect, due to Calegari and Emerton in \cite{CE}, is discussed in Remark \ref{calegariemerton}.  
\end{rem}

\subsection{Main tools}
\label{maintools}
Three main inputs are used to prove the theorems stated in $\S \ref{statementsandmethods}$:

\begin{itemize}
\item[(a)]
Cheeger-M\"{u}ller theorems.\smallskip

Let $\loc \rightarrow \M$ be a local system of metrized $\C$-vector spaces acted on equivariantly by an isometry $\sigma$ of finite order $p.$ A \emph{Cheeger-M\"{u}ller Theorem} is an identity 
\begin{equation*}
\tau_{\sigma}(\M, \loc) = RT_{\sigma}(\M, \loc).
\end{equation*} 

valid for some class of metrized local systems $\loc$ and some class of equviariant isometries $\sigma.$  When $\sigma = 1$ and $\loc$ is unimodular, this was proved by M\"{u}ller \cite{Mu2}.  

A general version of this theorem for $\sigma \neq 1$ is proven by Bismut and Zhang \cite{BZ2}.  This general version expresses the difference between $\log RT_{\sigma}(\M,\loc)$ and $\log \tau_{\sigma}(\M, \loc)$ in terms of auxillary differential geometric data on a germ of the fixed point set of $\sigma.$  The author proves in \cite{Lip1} that this difference equals zero in the cases to be studied in this paper.

Cheeger-M\"{u}ller theorems provide a bridge between the analytic expression $\tau_{\sigma}(\M,\loc)$ and the quantity $RT_{\sigma}(\M,\loc)$ which can concretely be related to the $\sigma$-module $H^{*}(\M,\loc)_{\mathrm{tors}}$ \cite[
Corollary 3.8]{Lip1}.  \medskip

\item[(b)]
Trace formula comparison.\smallskip

Using Cheeger-M\"{u}ller theorems, torsion in cohomology of arithmetic groups related by base change can be compared by instead comparing analytic torsions.  In the case of compact quotient, the logarithm of equivariant analytic torsion equals the spectral side of the twisted Arthur-Selberg trace formula for an appropriate test function (see $\S \ref{reptorsion}$).  The logarithm of analytic torsion equals the spectral side of the untwisted trace formula for a matching test function.  

Comparing these spectral quantities uses the methods of \cite{Lan} together with some local representation theory for $\PGL_2.$  In particular, we need to prove a ``fundamental lemma for the spherical unit" for tamely ramified base change of $\PGL_2.$  \medskip

\item[(c)]
The results of Bergeron and Venkatesh \cite{BV} on growth of untwisted analytic torsion for sequences of locally symmetric spaces with universal cover of fundamental rank 1.
\end{itemize}

Combining (a) with (b) proves an identity relating sizes of torsion cohomology for arithmetic groups related by base change (see Theorem \ref{cohomologycomparison}).  The resulting identity is consistent with implications of torsion base change functoriality (see $\S \ref{torsionfunctoriality}$).  

Combining (a) with (c) proves growth of ``Lefschetz numbers for torsion."  See Theorem \ref{refinedgrowthcohomology} for a more precise statement.

\subsection{Acknowledgements}

This paper is an outgrowth of the author's PhD thesis.  It owes its existence to the inspirational work of Bergeron-Venkatesh \cite{BV} and Calegari-Venkatesh \cite{CV}.  

The author thanks Jayce Getz, Les Saper, and Mark Stern for their helpful comments on drafts of this paper. 
 
The author would like to thank Nicolas Bergeron for many stimulating discussions on torsion growth and twisted endoscopy. 

Last but not least, the author would like to express his deep gratitude to his advisor, Akshay Venkatesh, for sharing so many of his ideas and for providing constant encouragement and support during the preparation of this work.

\subsection{Outline}
\label{outline}

\begin{itemize}
\item
In $\S \ref{torsionfunctoriality},$ we discuss base change functoriality over $\mathbb{Z}.$  Base change functoriality over $\mathbb{Z}$ predicts one of the main results of this paper, a relationship between the sizes of torsion subgroups on locally symmetric spaces related by base change.

\item
In $\S \ref{homogeneous},$ we discuss representation theoretic formulations of equivariant vector bundles, bundle-valued differential forms, local systems of vector spaces, and equivariant local sytems.  We see in $\S \ref{rational}$ how, for finite extensions $N/F,$ algebraic homomorphisms $\G \rightarrow R_{N/F} \GL(V)$ over $F$ give rise to local systems of $O_N$ modules over $M_U$ for appropriate compact open subgroups $U \subset \G(\finadele_F).$

\item
In $\S \ref{reptorsion},$ we will express the $\sigma$-twisted analytic torsion of $\localsystem \rightarrow \mathcal{M}_{\mathcal{U}}$ and the untwisted analytic torsion of $L_{\rho} \rightarrow M_U$  in purely representation theoretic terms.

\item
In $\S \ref{quat},$ we will use Langlands' representation theoretic statement of base change to prove an abstract matching theorem \ref{abstractmatching}, which will ultimately imply identities of the flavor 
$$\tau_{\sigma}(\M_\U, \localsystem) = \tau(M_U, L_{\rho})^p$$
of the aforementioned sample theorem.  In order to apply this matching theorem, we need to find instances of matching test functions, which will be the objective of $\S \ref{unrammatching}, \ref{tamematching}.$

\item
In $\S \ref{unrammatching},$ we will describe some circumstances under which the desired matching test functions can be found.  This matching only applies at places $v$ where $E_v/F_v$ is unramified.  Here is where we finally define and discuss parahoric level structure.

\item
In $\S \ref{tamematching},$ we prove a matching theorem at places $v$ where $E_v/F_v$ is tamely ramified.  In the final subsection $\S \ref{redefparahoric},$ we define tamely parahoric level structures, those level structures which occur in the matching theorem \ref{finalmatching} and the numerical cohomology comparison theorem \ref{cohomologycomparison}.  

\item
In $\S \ref{funct},$ we prove the numerical cohomology comparison theorem \ref{cohomologycomparison} for the local systems introduced in $\S \ref{weirdloc}.$ 

\item
In $\S \ref{growth},$ we use the main comparison theorem \ref{abstracttorsion} for analytic torsion together with \cite[Theorem 4.5]{BV} to prove that, for appropriate equivariant local systems $\loc$ and level structures $\U$ (see \ref{parahoric}), the twisted analytic torsion $\log \tau_{\sigma}(\M_\U, \loc) \sim \vol(\M_\U)^{\frac{1}{p}}.$   Combined with the results of \cite[
$\S 1$-$\S 5$]{Lip1} - which relate equivariant Reidemeister torsion to cohomology, a statement on asymptotic growth of cohomology is proven.  The results of \cite{BV} prove asymptotic growth of Reidemeister torsion, which one might loosely think of as an ``Euler characteristic for torsion in cohomology".  In the same vein, the results of $\S \ref{growth}$ prove ``asymptotic growth of Lefschetz numbers for torsion in cohomology." 
\end{itemize}

\subsection{Notation used throughout}
\label{commonnotation2}

This section compiles a list of frequently used notation.  The descriptions given are consistent with the most common usage of the corresponding symbols.  The reader should be warned, however, that within a given chapter or section, the below symbols might carry a slightly different meaning; such local changes of notation will be made clear as necessary.   

\subsubsection{Algebraic groups and representation theory notation}
\begin{itemize}
\item
$E/F$ denotes a cyclic Galois extension of number fields of odd prime degree $p$ with Galois group $\Galois = \langle \sigma \rangle.$  The rings $O_E, O_F$ denote the ring of integers of $E$ and $F$ respectively.

\item
For a field extension $N/F, \iota:N \hookrightarrow \C$ denotes a complex embedding of $N$ and the induced complex embedding of $F.$

\item
$D$ denotes a quaternion algebra over $F.$  

\item
$\G$ denotes an algebraic group over a number field $F.$  Unless otherwise specified, $\G$ denotes the adjoint group of $\underline{D}^{\times}.$

\item
$\U$ denotes a compact open Galois stable subgroup of $\G(\finadele_E)$ and $\K \subset \G(E_\R)$ denotes a Galois stable maximal compact subgroup  

\item
$U$ denotes a compact open subgroup of $\G(\finadele_F)$ and $K \subset \G(F_\R)$ denotes a maximal compact subgroup.

\item
$\overpi$ denotes a representation of $\G(\adele_E) \rtimes \Galois$ and $\pi$ denotes a representation of $\G(\adele_F).$

\item
$M_U:= \G(F) \backslash \G(\adele_F) / KU$ and $\M_\U := \G(E) \backslash \G(\adele_E) / \K \U.$

\item
$\widetilde{\rho}$ is a finite dimensional representation $R_{E/F} \G$ and $\rho$ is a finite dimensional representation of $\G.$

\item
$V_{\rho}$ denotes the local system of $O_N$-modules associated to a rational representation $\rho$ in the manner of $\S \ref{rational}.$

\item
$\overpi$ denotes a representation of $\G(\adele_E) \rtimes \Galois$ and $\pi$ denotes a representation of $\G(\adele_F).$

\item
$V'_{\rho} \rightarrow M_U$ denotes the local system of $F$-vector spaces associated to a representation $\G \rightarrow \GL(W)$ defined over $F.$  We let $V'_{\rho, O_F}$ denote the associated local system of $O_F$-modules, defined for appropriately chosen $U.$  See $\S \ref{rational}.$

\item
$V_{\rho} \rightarrow M_U$ denotes the local system of real (complex) vector spaces associated to a real (complex) representation $\rho: \G(F_\R) \rightarrow GL(W).$  See $\S \ref{greps}.$

\item
$r$ denotes the regular representation of $\G(\adele_F)$ on $L^2(\G(F) \backslash \G(\adele_F))$ and $\mathcal{R}$ denotes the regular representation of $\G(\adele_E) \rtimes \Galois$ on $L^2(\G(E) \backslash \G(\adele_E)).$

\item
For any representation $(\pi,W_{\pi})$ of $\G(\adele_F)$ and any compactly supported smooth measure $f dg$ on $\G(\adele_F),$ we let $\pi(f dg)$ denote the convolution operator 
$$\int_{\G(\adele_F)} f(g) \pi(g) dg \circlearrowleft W_\pi.$$  
Similarly for representations $\overpi$ of $\G(\adele_E).$

\item
We denote
\begin{eqnarray*}
&{}& \W := L^2(\G(E) \backslash \G(\adele_E)), \W_{\widetilde{\rho}} := L^2(\G(E) \backslash \G(\adele_E)) \otimes \widetilde{\rho} \\
&{}& W := L^2(\G(F) \backslash \G(\adele_F)), W_{\rho} := L^2(\G(F) \backslash \G(\adele_F)) \otimes \rho
\end{eqnarray*} 
for finite dimensional complex representations $\widetilde{\rho}$ and $\rho$ of $\G(F_\R), \G(E_\R)$ respectively.

\item
For a representation $\pi$ of a group $H$ and a representation $V$ of $H,$ we let $V[\pi]$ denote the $\pi$-isotypic subspace of $H,$ i.e. the image of the canonical evaluation map $\Hom_H(\pi,V) \otimes \pi \rightarrow V.$

\item
For a semisimple group $\G / \R$ and a maximal compact subgroup $K \subset \G(\R),$ we let $\theta$ denote the Cartan involution of $\G$ associated to $K;$ the fixed point set of $\theta$ acting on $\G(\R)$ equals $K.$ 

\item
If $F$ is a local field with ring of integers $O_F$ and maximal ideal $\mathfrak{m}_F,$ let $k_F$ denote the residue field $O_F / \mathfrak{m}_F.$

\item
For any finite dimensional representation $A$ of $\langle \sigma \rangle,$ we let $\langle A \rangle$ denote $\mathrm{tr}(\sigma | A).$  
\end{itemize}

\subsubsection{Reidemeister torsion notation}
\begin{itemize}
\item
$L \rightarrow M$ denotes a local system of projective $O_F,F, \mathbb{Q}, \mathbb{Z}, \R,$ or $\C$-modules, depending on the context.

\item
$\loc \rightarrow \M$ denotes a local system equivariant for the action of a finite group $\Gamma,$ usually $\Gamma = \langle \sigma \rangle$ with $\sigma^p = 1.$


\item
$\MS(X,L)$ denotes the Morse-Smale complex associated with a vector field $X$ on $M$ which is weakly gradient-like with respect to a fixed Morse function $f$ and which satisfies Morse-Smale transversality.  


\item
$RT(X, L)$ denotes the Reidemeister torsion of the Morse-Smale complex $\MS(X,L)$ for a vector field $X,$ satisfying Morse-Smale transversality, and a local system $L \rightarrow M,$ provided the Morse function $f$ and the implicit volume forms are understood.  $RT_{\sigma}(X,\loc)$ denotes the twisted Reidemesiter torsion of the Morse-Smale complex whenever $\loc \rightarrow \M$ is a $\langle \sigma \rangle$ equivariant local system.  We often suppress the $X$ and denote this by $RT(M,L).$  

\item
For an $R$-module $A$ acted on $R$-linearly by $\langle \sigma \rangle,$ we let $A^{\sigma - 1} := \{ a \in A: (\sigma - 1) \cdot a = 0 \}$ and $A^{P(\sigma)} = \{ a \in A: P(\sigma) \cdot a = 0 \}$ where $P(\sigma)$ denotes the $p$-cyclotomic polynomial $P(x) = x^{p-1} + x^{p-2} + ... + 1.$  Sometimes, we denote these by $A[\sigma - 1]$ and $A[P(\sigma)]$ as well.

\item
For an $R$-module $A$ acted on $R$-linearly by $\langle \sigma \rangle,$ we define $A' := A / (A[\sigma - 1] \oplus A[P(\sigma)]).$  Similarly, if $A^{\bullet}$ is a complex of $R$-modules acted on $R$-linearly by $\langle \sigma \rangle,$ we define $A'^{\bullet} :=  A^{\bullet} / (A^{\bullet}[\sigma - 1] \oplus A^{\bullet}[P(\sigma)]).$

\item
$\sum {}^{*}$ and $\prod {}^{*}$ respectively denote alternating sum and alternating product.  

\end{itemize}

\section{A priori predictions via torsion base change functoriality}
\label{torsionfunctoriality}

Calegari and Venkatesh \cite[$\S 2$]{CV} have conjectured an analogue of Langlands functoriality for mod $p$ and torsion cohomology of arithmetic locally symmetric spaces.  
In this section, we explain how one of our main results, Theorem $\ref{cohomologycomparison},$ is roughly predicted by their conjecture applied to base change.

\subsection{Base change functoriality over $\mathbb{Z}$}
\label{torsionbasechange}
For a more general discussion of functoriality over $\mathbb{Z},$ we refer the reader to \cite[$\S 2$]{CV}. \medskip

Let $F$ be any number field and $E/F$ a cyclic Galois extension of odd prime degree $p.$  Let $\G_1 / F$ be any group and $\G_2 = R_{E/F} \G_1.$  In accordance with functoriality over $\mathbb{Z},$ the diagonal map of $L$-groups
\begin{eqnarray*}
{}^L \G_1 &\xrightarrow{\phi}& {}^L \G_2 \\
(g,\sigma) &\mapsto& (g,...,g) \times \sigma
\end{eqnarray*}
is expected to give rise to a Hecke equivariant map $\phi_{*}$ on cohomology, torsion or otherwise.  For certain groups, such as $\G_1$ the unit group of a semisimple algebra over $F,$ this correspondence is known for characteristic zero cohomology.  

One main purpose of this paper is to unconditionally prove certain would-be-numerical consequences of the existence of the hypothetical map $\phi_{*}.$  The analogue of this program was carried out for Jacquet-Langlands functoriality for $\PGL_2$ in \cite{CV}.

\subsection{Conjectures on torsion base change}

This highly speculative section discusses implications of the existence of a base change map $\phi_{*},$ as above.  The discussion uses the language of Langlands' theory of base change, to be reviewed in $\S \ref{prelimbc}.$  \medskip

Let $D$ be a quaternion algebra over a number field $F.$  Let $\G$ denote the adjoint group of its group of units.  Let $E/F$ be a cyclic Galois extension with $\Galois = \langle \sigma \rangle.$  Let $\U \subset \G(\finadele_E)$ be Galois invariant and $U \subset \G(\finadele_F)$ be compact open subgroups.  Let $K \subset \G(F_\R)$ be a maximal compact subgroup and $\K \subset \G(E_\R)$ be a Galois invariant maximal compact subgroup.  Let $\M_\U = \G(E) \backslash \G(\adele_E) / \K \U$ and $M_U = \G(F) \backslash \G(\adele_F) / KU.$  For compact open subgroups $J \subset \G(\finadele_F)$ and $\mathcal{J} \subset \G(\finadele_E),$ let  $W^J = L^2(\G(F) \backslash \G(\adele_F) / J)$ and $\W^\mathcal{J} = L^2(\G(E) \backslash \G(\adele_E) / \mathcal{J}).$  

\begin{defn}
We say that $\mathbf{1}_{\U} d\tilde{u}$ and $\sum_U c_U \mathbf{1}_U du$ (finite sum)  are \emph{matching level structures} if 
$$\tr \{ \sigma | \W[\overpi]^\U \} = \sum_U c_U \dim W[\pi]^U$$
for every pair of representations $\overpi$ of $\G(\adele_E) \rtimes \Galois$ and $\pi$ of $\G(\adele_F)$ which match by base change (see $\S \ref{prelimbc}$). If $E'/F$ is a second cyclic Galois extension and $\U' \subset \G(\adele_{E'})$ and $\U \subset \G(\adele_E)$ both match a common level structure, we say that $\U, \U'$ are \emph{related}. 
\end{defn}

\begin{lem}
Matching level structures satisfy the following identity of traces in cohomology
\begin{equation}
\tr \{ \sigma | H^{*}(\M_{\U}, \loc_{\mathbb{C}}) \} = \sum_U c_U \dim H^{*}(M_U, L_{\mathbb{C}}).
\end{equation}
for local systems $\loc$ and $L$ associated to compatible representations $\widetilde{\rho}$ of $R_{E/F} \G \rtimes \Galois$ and $\rho$ of $\G$ (see $\S \ref{compatibleloc}$ for a discussion of matching representations and matching local systems).
\end{lem}

\begin{proof}
We can decompose $H^{*}(M_U, L) = \bigoplus \dim W[\pi]^U \cdot H^{*}(\pi_{\infty} \otimes \rho)$ in accordance with Matsushima's formula \cite[VII.5.2]{BW}, where $H^{*}(\pi_{\infty} \otimes \rho)$ denotes $(\g,K)$-cohomology; this is a representation theoretic incarnation of Hodge theory.  Because the level structures match, we are reduced to proving that 
$$\tr\{ \sigma | H^{*}(\overpi_{\infty} \otimes \widetilde{\rho}) \} = \dim H^{*}(\pi_{\infty} \otimes \rho)$$
for pairs of representations $\overpi$ of $\G(\adele_E) \rtimes \Galois$ and $\pi$ of $\G(\adele_F)$ which match by local base change.  In this particular situation, $\G(E_\R)$ is isomorphic to $\G(F_\R)^p$ and that $\overpi_{\infty} \cong \pi_{\infty}^{\boxtimes p},$ where $\sigma$ acts by cyclic permutation.  By the K\"{u}nneth formula for $(\g,K)$-cohomology, $H^{*}(\overpi_{\infty}) \cong H^{*}(\pi_{\infty})^{\otimes p}$ (graded tensor product), where $\sigma$ again acts by cyclic permutation.  The result then follows by the elementary fact that for any finite dimensional $V,$ 
$$\tr \{ \text{ cyclic permutation } | V^{\otimes p} \} = \dim V.$$ 
\end{proof}

\begin{cor}
If $\U \subset \G(\adele_E)$ and $U' \subset \G(\adele_E)$ are related level structures, then 
$$\tr \{ \sigma | H^{*}(\M_\U, \loc_{\mathbb{C}}) \} = \tr \{ \sigma | H^{*}(\M'_{\U'}, \loc'_{\mathbb{C}}) \}.$$
\end{cor}

We optimistically conjecture that an analogous conjecture is true of torsion cohomology.

\begin{conj}[Galois structure] \label{galoisstructure}
For matching rationally acyclic local systems $\loc \rightarrow \M_U$ and $\loc' \rightarrow \M_{\U'},$ the graded $\mathbb{Z}[\sigma]$-modules $H^{*}(\M'_{\U'}, \loc')$ and $H^{*}(\M_\U, \loc)$ are isomorphic.
\end{conj}

Let $E' = F \times ... \times F / F$ be a split extension of degree $p$ and $E / F$ is a cyclic Galois extension of degree $p.$  For simplicity, suppose that there are compact open subgroups $\U \subset \G(\finadele_E), J \subset \G(\finadele_F)$ for which $\mathbf{1}_\U du$ matches $\mathbf{1}_J dj$; examples of this sort are constructed in $\S \ref{unrammatching}.$  If we simply take $\U' = J \times ... \times J,$ then $\U$ and $\U'$ are related and $\M'_{\U'} = M_J \times ... \times M_J.$  Suppose that $\loc \rightarrow \M_\U$ and $L \rightarrow M_J$ are matching (rationally acyclic) local systems (see $\S \ref{compatibleloc}$).  Then Conjecture $\ref{galoisstructure}$ predicts that 
$$H^{*}(\M_\U, \loc) \cong_{\mathbb{Z}[\sigma]} H^{*}(\M'_{\U'}, L^{\boxtimes p}) =  H^{*}(M_J \times ... \times M_J, L^{\boxtimes p}) \cong \left(H^{*}(M_J,L) \right)^{\widehat{\otimes} p},$$
where $\widehat{\otimes} p$ denotes the $p$-fold left derived tensor product.  For example, suppose the $\ell$-primary part of $H^{*}(M_J, L)$ is isomorphic to $\mathbb{Z} / \ell \mathbb{Z}$ in degree $d,$ generated by $c \in H^d(M_J, L).$  Conjecture \ref{galoisstructure} predicts the existence of a graded $\mathbb{Z}[\sigma]$ submodule $\widetilde{C} \subset H^{*}(\M_U, \loc)$ isomorphic to $H^*( (\mathbb{Z} \xrightarrow{\ell} \mathbb{Z})[d]^{\otimes p});$ the action of $\sigma$ on the latter group is induced by cyclic permutation of the tensor factors. 

\begin{rem} \label{derivedtensorproduct}
One computes that $H^*( (\mathbb{Z} \xrightarrow{\ell} \mathbb{Z})[d]^{\otimes p})$ is isomorphic, as a graded $\mathbb{Z}[\sigma]$-module, to an exterior algebra on the $\mathbb{Z} / \ell\mathbb{Z}$-vector space $(\mathbb{Z} / \ell \mathbb{Z}) [\sigma] / \langle \sigma^{p-1} + ... + \sigma + 1 \rangle$ starting in degree $p(d-1) + 1.$  
\end{rem}

We mention, in passing, an even more speculative conjecture, pertaining to Hecke equivariance of this correspondence on cohomology.  

\begin{conj}[Hecke equivariance of base change]
Suppose $c, \widetilde{C}$ are as above. Suppose that $c$ is a $\G$-Hecke eigenclass, i.e. for any $v$ where $\G / F_v$ is split and $J_v$ is hyperspecial, then 
$$T_{\rho}(c) = a_{\rho} c$$
for all $T_{\rho} \in \mathcal{H}_v \cong \mathrm{Rep}( {}^L \G).$  Then for any $\tilde{c} \in \widetilde{C},$
$$T_{\tilde{\rho}} c = a_{\tilde{\rho} \circ \phi} \tilde{c}.$$ 
\end{conj}

\subsection{A numerical consequence to be expected of torsion base change}

Let $\U, \U', \loc, L, E, E', F$ be as in the previous section. \medskip
 
For Galois-equivariant, rationally acyclic local systems $\loc \rightarrow \M_\U$ as above, where $\Galois = \langle \sigma \rangle,$ we consider the alternating product
$$R_{\sigma}(\loc) := \prod {}^{*} \frac{|H^i(\M_\U, \loc)^{\sigma - 1}|}{ |H^i(\M_\U, \loc)^{P(\sigma)}|^{\frac{1}{p-1}}}$$
In support of Conjecture \ref{galoisstructure} and Remark \ref{derivedtensorproduct}, we compute in \cite[
Lemma 5.3]{Lip1} that  
\begin{equation}
R_{\sigma}(\loc) =  R_{\sigma}(L^{\boxtimes p}) = R(L)^p, R(L) := \prod {}^{*} |H^i(M_J, L)|,
\end{equation} 
at least up to powers of $p.$  \bigskip

The invariant $R_{\sigma}(\loc)$ is very closely related to the twisted Reidemeister torsion of $\loc$ (see \cite[
$\S 1.3$]{Lip1}).  It is a miraculous fact (see \cite[Theorem 0.2]{BZ2}, \cite[
Theorem 1.23]{Lip1}) that the twisted Reidemeister torsion is closely related to the twisted analytic torsion $\tau_{\sigma}(\loc),$ an equivariant spectral invariant of the metrized, equivariant local system $\loc$ (described in \cite[
$\S 1.1$]{Lip1}).  The main content of this paper is comparing $\tau_{\sigma}(\loc)$ to the untwisted analytic torsion $\tau(L)$ of a matching (see $\S \ref{compatibleloc}$) local system.  A prototypical example:  

\begin{sample}
Let $E/F$ be everywhere unramified.  Let $\U \subset \G(\finadele_E)$ be a parahoric level structure (see Definition \ref{parahoric}) and $U \subset \G(\finadele_F)$ an associated level structure (see Definition \ref{associatedlevel}).  Then  
$$\tau_{\sigma}(\M_\U, \loc) = \tau(M_U,L)^p$$  
\end{sample}

(see Theorem \ref{abstracttorsion} combined with Theorem \ref{finalmatching} for a more general statement).  By applying appropriate versions, both twisted and untwisted, of the Cheeger-M\"{u}ller theorem  (see \cite[
$\S 1, \S 2$]{Lip1}), we arrive at equalities
$$R_{\sigma}(\M_\U, \loc) \sim \tau_{\sigma}(\M_\U, \loc) = \tau(M_U,L)^p = R(M_U, \loc)^p,$$
where $\sim$ denotes equality up to powers of $p.$  This reasoning is applicable to those matching pairs $\loc, L$ of local systems described in $\S \ref{weirdloc}.$

\section[Representations and vector bundles]{Representations and equivariant vector bundles}
\label{homogeneous}

The purpose of this chapter is to discuss representation theoretic formulations of equivariant vector bundles, bundle-valued differential forms, local systems of vector spaces, equivariant local sytems, and integral structures on these local systems.

\subsection{Equivariant vector bundles}
\label{kreps}

Let $\G$ be a semisimple algebraic group over $\R.$  Let $G = \G(\R)$ with maximal compact subgroup $K$ and let $X = G/K.$  Let $e = eK$ be a fixed choice of basepoint.

\begin{lem} \label{hom}
There is an equivalence of categories 
$$\{ G \text{-equivariant real (complex) vector bundles } V \rightarrow X \} \rightarrow \{ \text{real (complex) representations of } K \}$$
\end{lem}

\begin{proof}
(See \cite[$\S 3.2$]{BV})  To any $G$-equivariant vector bundle $V \rightarrow X,$ the fiber $V_e$ admits a $K$-representation.  \medskip

Given any $K$-representation $\rho: K \rightarrow GL(W),$ we can form $V = (G \times W) / K$ where $K$ acts diagonally by $(g,w) \cdot k = (gk, \rho(k^{-1}) w).$  Projection onto the first factor makes $V$ an equivariant vector bundle with action $g' \cdot (g,w) = (g'g, w)$ covering the left translation action of $G$ on $X.$  \medskip 

These two constructions define inverse equivalences of categories.
\end{proof}

\subsection{$G$-representations and flat homogeneous vector bundles}
\label{greps}

The image under the inverse equivalence from Lemma \ref{hom} of those representations $\rho: K \rightarrow GL(W)$ which arise as the restriction of a representation $\rho: G \rightarrow GL(W)$ can be described as $G$-equivariant vector bundles equipped with a $G$-compatible flat connection.

\begin{defn}
Let $V \rightarrow X$ be a $G$-equivariant vector bundle equipped with connection $\nabla_V.$  We say that $\nabla_V$ is \emph{$G$-compatible} if parallel transport along every path $\gamma$ commutes with translation by $G$: $g \cdot P_{\gamma} = P_{g \cdot \gamma}.$
\end{defn}

\begin{lem} \label{flathom}
There is an equivalence of categories
\begin{eqnarray*}
\{ (V, \nabla_V): V &=&  G \text{-equivariant real (complex) vector bundle over } X, \\
 \nabla_V &=& G \text{-compatible flat connection }  \}
\end{eqnarray*} 
$$\rightarrow \{ \text{real (complex) representations } \rho: G \rightarrow GL(W)\}.$$
\end{lem}

\begin{proof}
Choose any path $\gamma$ from $e$ to $gK.$  Let $P_{\gamma}$ denote parallel translation along $\gamma$ with respect to the connection $\nabla_V.$  Note that this is independent of the path chosen because $X$ is contractible and $\nabla_V$ is flat.  For any $v \in V_e,$ let 
$$g \cdot v = P_{\gamma}^{-1} \circ t_g(v),$$
where $t_g$ denotes the action on $V$ covering left translation by $g.$  Let $\gamma'$ be any path from $e$ to $hK.$  We compute that
\begin{eqnarray*}
g \cdot (h \cdot v) &=& P_{\gamma}^{-1} \circ t_g \circ P_{\gamma'}^{-1} \circ t_h (v) \\
&=& P_{\gamma}^{-1} \circ (t_g P_{\gamma'}^{-1} t_g^{-1}) \circ t_g t_h(v) 
\end{eqnarray*}
By assumption, the $G$-action respects parallel transport, which is to say that 
$$t_g P_{\gamma'}^{-1} t_g^{-1} = P_{g \cdot \gamma'}^{-1}$$
as isomorphisms $V_{ghK} \rightarrow V_{gK}.$  Therefore, the above computation further simplifies to 
\begin{eqnarray*}
&=& P_{\gamma}^{-1} \circ P_{g \cdot \gamma'}^{-1} \circ t_g t_h(v) \\
&=& (P_{g \cdot \gamma'} \circ P_{\gamma})^{-1} \circ t_{gh}(v) \\
&=& (P_{(g \cdot \gamma') \cdot \gamma})^{-1} \circ t_{gh}(v) \\
&=& (gh) \cdot v.
\end{eqnarray*}
In the above, $(g \cdot \gamma') \cdot \gamma$ denotes the concatenation of the paths $g \cdot \gamma'$ and $\gamma.$  Thus, any pair $(V \rightarrow X, \nabla_V)$ gives rise to a representation $\rho_V: G \rightarrow GL(W).$  Furthermore, by choosing the constant path from $e$ to $e$, we see that $\rho_V|_K$ is the isotropy representation of $K$ on $V_e.$  \medskip

Conversely, suppose we are give a representation $\rho: G \rightarrow GL(W).$  This gives rise to a unique $G$-equivariant vector bundle $V_{\rho} \rightarrow X$ by the ``mixing construction" of lemma \ref{hom}.  Namely, $V_{\rho} = (G \times W)/K \rightarrow X$ with bundle projection given by the first coordinate projection, where $K$ acts diagonally by $(g,v) \cdot k = (gk, \rho(k^{-1}) v).$  \smallskip

Consider the trivial bundle $X \times W \rightarrow X.$  The map $\phi: (gK, v) \mapsto (g, g^{-1} v) / \sim$ defines a $G$-equivariant isomorphism 
$$\begin{CD} 
X \times W @>\phi>> V_{\rho}  \\
@VVV                              @VVV \\
X @>=>> X
\end{CD}$$
with inverse isomorphism $\phi^{-1}: (g,v) \mapsto (g, gv).$  We define a connection $\nabla_V$ by $\phi \nabla^{\mathrm{triv}} \phi^{-1}$ where $\nabla^{\mathrm{triv}}$ denotes the trivial connection on the trivial vector bundle $X \times W \rightarrow X.$  Evidently, $\nabla_V$ is flat because parallel transport, defined through the parallel transport of $\nabla^{\mathrm{triv}},$ is independent of path.  Furthermore, $\nabla_V$ is $G$-compatible: for any path $\gamma_{hK,gK}$ from $gK$ to $hK$ and any vector $v \in W,$ parallel transport is given by
$$P_{\gamma_{hK,gK}}((g,v)) = (h,v) \in V_{\rho, hK};$$   
this parallel transport clearly commutes with the action of $G.$ \medskip

The above two constructions are mutually inverse isomorphisms.  
\end{proof}

\subsection{Rational and integral structures on local systems}
\label{rational}

Let $W$ be an $N$-vector space with $N/F$ a finite extension of number fields.  Let $\G$ be a semisimple group over $F.$  Let $\rho: \G \rightarrow R_{N/F} \GL(W)$ be an algebraic representation \emph{defined over $F$}.  Let $U \subset \G(\finadele_F)$ be a compact open subgroup and $K \subset \G(F_\R)$ be a maximal compact subgroup with $X = \G(F_\R) / K.$  We can form an associated local system of $N$-vector spaces
$$V'_{\rho} = \G(F) \backslash \left(  (\G(\adele_F) / KU) \times W \right) \rightarrow M_U = \G(F) \backslash \G(\adele_F) / KU$$
where $g \cdot (x,w) = (gx, \rho(g)w)$ for $g \in \G(F).$  This is well-defined because the stabilizer of $\G(F)$ acting on $gKU$ is $\G(F) \cap gKUg^{-1};$ this stabilizer is contained in $\G(F),$ which certainly preserves the fiber $W$ over $gKU.$ \medskip

Now suppose that $\mathcal{O} \subset W$ is an $O_N$-lattice.  The group $\G(\finadele_F)$ acts through $\rho$ on the space of $O_N$ lattices of $W,$ and we suppose that $U$ stabilizes $\mathcal{O}.$  Consider the local system of $O_N$-lattices over $X \times \G(\finadele_F) / U$ given by 

$$V'_{\rho, O_N} = \{ (x,gU,v): v \in \rho(g) \mathcal{O} \} \rightarrow X \times \G(\finadele_F) / U$$ 

with the bundle projection given by projection onto the first two factors.  The group $\G(F)$ acts on $V'_{\rho}$ diagonally: $g \cdot (x,gU,v) = (gx, gu, \rho(g) v).$  We let 

$$V''_{\rho,O_N} = \G(F) \backslash V'_{\rho, O_N} \rightarrow M_U$$

denote the quotient.  

\begin{lem} \label{integrallocalsystem}
The bundle $V''_{\rho,O_N} \rightarrow M_U$ enjoys the following properties:

\begin{itemize}
\item[(1)]
$V''_{\rho,O_N}$ is a local system of $O_N$-modules.  

\item[(2)]
$V''_{\rho, O_N} \otimes_{O_N} N = V'_{\rho}.$
\end{itemize}
\end{lem}

\begin{proof}
Property (2) follows because $(\rho(g) \mathcal{O})_N = \mathcal{O}_N = W.$  \medskip

Note that $V''_{\rho, O_N}$ is a well-defined sheaf whose stalks are $O_N$-modules.  Indeed, the stabilizer in $\G(F)$ of $(x, gU)$ is contained in $\G(F) \cap g U g^{-1}.$  But if $v \in \rho(g) \mathcal{O},$ then

$$\rho(gug^{-1}) v \in \rho(g) \rho(U) \rho(g)^{-1} \rho(g) \mathcal{O} = \rho(g) \rho(U) \mathcal{O} \subset \rho(g) \mathcal{O}$$

because $\rho(U)$ stabilizes $\mathcal{O}$ by assumption.  Because $V'_{\rho}$ is locally constant, property $(2)$ implies that $V''_{\rho,O_N}$ is locally constant as well.
\end{proof}

Let $\iota: N \hookrightarrow \C$ be a complex embedding.  

\begin{lem} \label{compatwithflathom}
Let $V'_{\rho, \iota} := V'_{\rho} \otimes_\iota \C \rightarrow M_U,$ a local system of $\C$-vector spaces.  Then $\widetilde{V'_{\rho, \iota}}$ is $\G(F_\R)$-equivariantly isomorphic to the flat bundle $V_{\rho,\iota}$ arising from the representation $\rho_{\iota}: \G(F_\R) \rightarrow GL(W_{\iota}).$  
\end{lem}

\begin{proof}
Clearly,
$$V'_{\rho, \iota} = \G(F) \backslash ((\G(\adele_F) / KU) \times W_{\iota}) \rightarrow M_U.$$
But it follows directly from the construction $\rho \leadsto V_{\rho}$ of Lemma \ref{flathom} that $\widetilde{V'_{\rho,\iota}}$ is $\G(F_\R)$-equivariantly isomorphic to the vector bundle $(\G(F_\R) \times \G(\finadele_F) / U \times W_{\iota}) / K$  via the isomorphism

$$\widetilde{V'_{\rho, \iota}} \rightarrow \widetilde{V_{\rho, \iota}}$$  
\begin{equation} \label{isomorphism} 
(gK,hU,v) \mapsto (gx, hU, \rho(g)^{-1}v).
\end{equation}
This completes the proof.
\end{proof}

\subsection{Bundle-valued $L^2$-differential forms}

\subsubsection{Global sections of $V_{\rho,\iota}$}
It will be convenient to recast the global sections of $V'_{\rho, \iota}$ using the isomorphism of Lemma \ref{compatwithflathom}.  Consider the bundle 
$$V_{\rho, \iota} := \left( \G(F) \backslash \G(\adele_F) / U \times W_{\iota} \right) /K ,$$
where $(g,w) \cdot k = (gk, \rho(k)^{-1} w).$ Let $\widetilde{V_{\rho,\iota}}$ and $\widetilde{V'_{\rho, \iota}}$ denote the pullbacks of $V_{\rho, \iota}$ and $V'_{\rho, \iota}$ to $X \times \G(\finadele_F) /U.$  According to Lemma \ref{compatwithflathom}, the vector bundles $\widetilde{V_{\rho,\R}}$ and $\widetilde{V'_{\rho, \R}}$ are $\G(F_\R)$-equivariantly isomorphic via the isomorphism 
$$\widetilde{V'_{\rho, \iota}} \rightarrow \widetilde{V_{\rho, \iota}}$$  
\begin{equation} \label{isomorphism} 
(gK,hU,v) \mapsto (gx, hU, \rho(g)^{-1}v).
\end{equation}
The global sections of $V_{\rho, \iota} = \G(F) \backslash \widetilde{V_{\rho,\iota}}$ can be described as functions
$$f:  \G(F) \backslash \G(\adele_F) / U \rightarrow W_{\iota} \text{ satisfying } f(xk) = \rho(k)^{-1} f(x) \text{ for all } x \in \G(F) \backslash \G(\adele_F) / U, k \in K.$$

\subsubsection{Normalization of metrics on $V_{\rho, \iota}$}

We refer to \cite[$\S 3.4$]{BV} for a discussion of the normalizations used here. \medskip

Let $W$ be a vector space over $N,$ assumed to be a finite extension of the number field $F.$  Let $\iota: N \hookrightarrow \C$ denote a fixed complex embedding.  Let $\rho: \G \rightarrow R_{N/F} \GL(W)$ be a homomorphism of algebraic groups over $F.$  Let $K \subset \G(F_\R)$ be a maximal compact subgroup.  We let $\theta$ be the Cartan involution corresponding to $K$ with associated decomposition $\mathfrak{g} := \mathrm{Lie}(\G(F_\R)) = \mathfrak{k} \oplus \mathfrak{ p}$ into the $+1$-eigenspace $\mathfrak{k} = \mathrm{Lie}(K)$ and the $-1$-eigenspace $\mathfrak{p}.$  \medskip

By Weyl's unitary trick, the irreducible representation $\rho:\G(F_\R) \rightarrow GL(W_\iota)$ corresponds to a unique representation, which we also call $\rho,$ of the compact dual group $U$ of $G;$ more precisely $U$ is the normalizer in $\G(F_{\mathbb{C}})$ of the real Lie algebra $\mathfrak{k} \oplus i \mathfrak{p}.$  There is a unique Hermitian metric on $W_{\iota}$ which is $U$-invariant, up to scaling.  Fix such a choice of metric $\langle \cdot, \cdot \rangle_0.$  We use this choice to define metrics $||\cdot||'$ and $||\cdot||$ on the bundles $V'_{\rho,\iota}, V_{\rho, \iota}$ by 
$$||(g_{\infty}K, hU,v)||'^2 = ||v||^2_0 \text{ for } (g,v) \in V'_{\rho, \R}$$    
$$||(g_{\infty}K, hU,v)||^2 := ||\rho(g_{\infty}) v||^2_0 \text{ for } (g,v) \in V_{\rho,\R}.$$
The isomorphism (\ref{isomorphism}) becomes an isometry with respect to these compatible choices of metrics.

\subsubsection{Differential forms}
\begin{lem} \label{diffforms}
We can isometrically -  with its metric induced by the invariant metric $\langle \cdot, \cdot \rangle_0$ on $\rho$ - identify global sections of $\Omega^j(M_U,V_{\rho, \iota}),$ the $V_{\rho,\iota}$-valued $L^2$ differential forms on $\M_U,$ with
$$\Hom_K(\wedge^j \mathfrak{p}, W^U \otimes \rho_{\iota}),$$
where $W^U =  L^2(\G(F) \backslash \G(\adele_F) / U).$
\end{lem}

\begin{proof}
There is a natural isomorphism given by
$$\Omega^j(\mathcal{M}, V_{\widetilde{\rho}_{\iota}}) \xrightarrow{\alpha_j} \Hom_{\mathcal{K}}(\wedge^j \mathcal{\mathfrak{p}}, W^U \otimes \widetilde{\rho}_{\iota} )$$
$$\omega \otimes X \mapsto f_v(\omega \otimes X)(g) =  \{ g^{*} \omega \otimes g^{-1} X \}_e(v).$$
The isometry property is a matter of unraveling definitions.
\end{proof}

\subsubsection{Kuga's Lemma} \label{kugaslemma}
We recall the identification of Lemma \ref{diffforms}:
$$\Omega^j(\mathcal{M}, V_{\widetilde{\rho}_{\iota}}) \xrightarrow{\alpha_j} \Hom_{\mathcal{K}}(\wedge^j \mathcal{\mathfrak{p}}, W^U \otimes \widetilde{\rho}_{\iota} )$$
$$\omega \otimes X \mapsto f_v(\omega \otimes X)(g) =  \{ g^{*} \omega \otimes g^{-1} X \}_e(v).$$
There is a differential operator $dd^{*} + d^{*} d$ which acts on the left side; this is the $j$-form Laplacian corresponding to the metric $\langle \cdot, \cdot \rangle_0$ chosen for $\widetilde{\rho}_{\iota}.$  The Casimir operator $C$ acts on the right side and is defined independent of the metric.
\begin{lem}[Kuga's Lemma] \label{kuga}
The operators $d d^{*} + d^{*} d$ and $C$ correspond via the above identification $\alpha_j.$  More formally, for every $j,$  
\begin{equation} \label{laplacianvscasimir}
\alpha_j( (dd^{*} + d^{*} d) \omega) = C \cdot \alpha_j(\omega)
\end{equation}
for all $\omega \in \Omega^j(\mathcal{M}, V_{\widetilde{\rho}_{\iota}}).$ 
\end{lem}
\begin{proof}
See \cite[II.2]{BW}.
\end{proof}


\subsubsection{Equivariant local systems and their associated differential forms}

Suppose that $\sigma$ acts on a group $\mathbf{H} / F$ with $\sigma^p = 1.$  

\begin{rem}
The situation we have in mind is $\mathbf{H} = R_{E/F} \G_E,$ where $E/F$ is a cyclic, degree $p$ Galois extension and $\langle \sigma \rangle = \Galois.$   
\end{rem}

We can form the group $\mathbf{H} \rtimes \langle \sigma \rangle.$  Representations $\rho: \mathbf{H} \rtimes \langle \sigma \rangle \rightarrow R_{N/F} \GL(W)$ give rise to equivariant local systems of $O_N$-modules $V'_{\rho}$ over $M_U = \mathbf{H}(F) \backslash \mathbf{H}(\adele_F) / KU,$ provided the maximal compact subgroup $\K \subset \mathbf{H}(F_\R)$ and $\U \subset \mathbf{H}(\finadele_F)$ are $\sigma$-stable and that $\U$ is sufficiently small (see Lemma \ref{integrallocalsystem}).  The $\mathbf{H}(F_\R)$-equivariant isometry $\widetilde{V'_{\rho, \iota}} \rightarrow \widetilde{V_{\rho, \iota}}$ from \eqref{isomorphism} is in fact an $\mathbf{H}(F_\R) \rtimes \langle \sigma \rangle$-equivariant isometry.  Finally,

\begin{lem}
The identification
$$\Omega^j(\mathcal{M}, V_{\widetilde{\rho}_{\iota}}) \rightarrow \Hom_{\mathcal{K}}(\wedge^j \mathcal{\mathfrak{p}}, W^U \otimes \widetilde{\rho}_{\iota} )$$
$$\omega \otimes X \mapsto f_v(\omega \otimes X)(g) =  \{ g^{*} \omega \otimes g^{-1} X \}_e(v)$$
from the proof of Lemma \ref{diffforms} is equivariant for the action of $\sigma$ by pullback of the left side and by $\sigma$ on $\wedge^j \mathfrak{p}$ and by $(\sigma^{-1})^{*} \otimes \rho(\sigma)^{-1}$ on $W^U \otimes \widetilde{\rho}_{\iota}.$ 
\end{lem}

\begin{proof}
We compute that

\begin{eqnarray*}
f_v(\sigma^{*} \omega \otimes \sigma^{-1} X )(g) &=& \{ g^{*} \sigma^{*}\omega \otimes g^{-1} \sigma^{-1} X \}_e(v) \\
&=& \{ \sigma^{*} (g^{\sigma^{-1}})^{*} \omega \otimes {\sigma}^{-1} (g^{-1})^{\sigma^{-1}} X \}_e(v) \\
&=& \{ \sigma \cdot g^{\sigma^{-1}} \cdot ( \omega \otimes X) \}_e ( v) \\
&=&  \{ g^{\sigma^{-1}} \cdot (\omega \otimes X) \}_e (\sigma \cdot v) \\
&=& f_{\sigma \cdot v}(\omega \otimes X)^{\sigma^{-1}}(g).
\end{eqnarray*}
\end{proof}

\section[Analytic torsion via representation theory]{Recasting equivariant torsion representation theoretically}
\label{reptorsion}

Building towards a proof of our main theorem, in this section we express the equivariant zeta functions $\zeta_{j,\loc_{\widetilde{\rho}},\Gamma}$ in purely representation theoretic terms.  For the purposes of this calculation, there is no need to restrict our attention to quaternion algebras. \bigskip

Let $\G$ be a semisimple group over a number field $F.$   Let $E/F$ be a cyclic Galois extension of prime degree $p$ with Galois group $\Galois = \langle \sigma \rangle.$  Let $N/F$ be a finite extension and $\iota: N \hookrightarrow \C$ be a complex embedding. 

\newtheorem*{trace}{Notation for trace}

\begin{trace} \label{tracenotation}
For a $\sigma$-module $A,$ we let $\langle A \rangle$ denote $\tr \{ \sigma | A  \}.$
\end{trace}

Let $\mathcal{U}$ be a $\sigma$-stable compact open subgroup of $\G(\mathbb{A}_E^{\mathrm{fin}})$ and let $\mathcal{K}$ be a $\sigma$-stable maximal compact subgroup of $\G(E_{\mathbb{R}}).$  Let the manifold $\mathcal{M} = \G(E) \backslash \G(\mathbb{A}_E) / \mathcal{K} \; \mathcal{U}$ be equipped with the invariant Riemannian metric induced by the Killing form of $\mathrm{Lie}(\G(E_\R));$ this Riemannian metric is preserved by $\sigma.$  \smallskip

Let $\widetilde{\rho}: R_{E/F} \G_E \rtimes \Galois \rightarrow R_{N/F} \GL(V_{\widetilde{\rho}})$ be a representation defined over $F.$  We let $V_{\widetilde{\rho},\iota}$ denote the associated $\G(E_\R) \rtimes \Galois$-equivariant complex vector bundle over $\G(E_{\mathbb{R}}) / \mathcal{K}.$ \bigskip

Let $\W =  L^2(\G(E) \backslash \G(\mathbb{A}_E))$ with its evident  $\G(\adele_E) \rtimes \Galois$ action.  Let $\W_{\widetilde{\rho}, \iota} =  L^2(\G(E) \backslash \G(\mathbb{A}_E)) \otimes \widetilde{\rho}_{\iota};$ this is a representation of $\G(\adele_E) \rtimes \Galois$ under the diagonal action, where $g \in \G(\adele_E)$ acts on $\widetilde{\rho}_\iota$ by $\widetilde{\rho}(g_{\infty}).$  Similarly, we let $\W^\U = L^2(\G(E) \backslash \G(\mathbb{A}_E) / \mathcal{U}),$ which is a $\G(E_{\mathbb{R}}) \rtimes \Galois$ representation (via the diagonal action).  Assume that $\mathcal{M}$ is compact.  Then $\W$ decomposes discretely with isotypic decomposition 
$$\W = \bigoplus_{\overpi} \W[\overpi]$$  
where $\overpi$ ranges over all irreducible unitary representations of $\G(\adele_E) \rtimes \Galois.$  Correspondingly,
$$\W^\U = \bigoplus_{\overpi} \W[\overpi]^\U.$$
By Kuga's Lemma (see $\S \ref{kugaslemma}$), the Laplacian $\Delta_j$ acts on the $\overpi$-isotypic subspace of $\Omega^j(\mathcal{M}, V_{\widetilde{\rho}, \iota}),$ that is on $\Hom_{\K}(\wedge^j \widetilde{\mathfrak{p}}, \W[\overpi]^\U \otimes \widetilde{\rho}_{\iota}),$ by the scalar $\lambda_{\widetilde{\rho}_{\iota}} - \lambda_{\overpi_{\infty}}.$ Here, for an admissible irreducible representation $T$ of $\G(E_{\mathbb{R}}), \lambda_T$ denotes its Casimir eigenvalue.  According to the identification from Lemma \ref{diffforms}, we can thus decompose 
\begin{equation} \label{decomposingforms}
\Omega^j(\mathcal{M},V_{\widetilde{\rho},\iota})^{\Delta_j = \lambda} = \bigoplus_{\lambda_{\widetilde{\rho}_{\iota}} - \lambda_{\overpi_{\infty}} = \lambda} \Hom_{\mathcal{K}}(\wedge^j \widetilde{\mathfrak{p}}, \W[\overpi]^\U \otimes \widetilde{\rho}_\iota).
\end{equation}
\begin{defn} \label{essential}
We call an irreducible representation $\overpi$ of $\G(\adele_E) \rtimes \Galois$ \emph{essential} if $\overpi|_{\G(\adele_E)}$ is irreducible.  We call it \emph{inessential} otherwise. 
\end{defn}

If $\overpi$ is inessential, then $\overpi|_{\G(\adele_E)}$ decomposes as a direct sum of $p$ (irreducible) representations $V_1 \oplus ... \oplus V_p$ where $\overpi(\sigma)$ cyclically permutes the $V_i.$  In particular, because no summand $V_i$ is preserved, it readily follows that 
$$\left\langle \Hom_{\mathcal{K}}(\wedge^j \widetilde{\mathfrak{p}}, \W[\overpi]^\U \otimes \widetilde{\rho}_\iota) \right\rangle = 0 \text{ for inessential } \overpi.$$
This implies that,
$$\left\langle \Omega^j(\mathcal{M}, V_{\widetilde{\rho}_\iota} )^{\Delta_j = \lambda} \right\rangle = \sum_{\lambda_{\widetilde{\rho}_\iota} - \lambda_{\overpi_{\infty}} = \lambda, \overpi \text{ essential }} \left\langle  \Hom_{\mathcal{K}}(\wedge^j \widetilde{\mathfrak{p}}, \W[\overpi]^\U \otimes \widetilde{\rho}_\iota) \right\rangle .$$
Suppose that $\overpi = \overpi_{\infty} \otimes \overpi_{\mathrm{fin}}$ as a representation of $\G(\adele_E).$  Suppose further that \emph{the action of $\sigma$ on $\overpi$ factorizes as $\overpi(\sigma) = \overpi_{\infty}(\sigma_{\infty}) \otimes \overpi_{\mathrm{fin}}(\sigma_{\mathrm{fin}})$}. 

\begin{rem} \label{factorization}
Let $\G$ be the adjoint group of $\underline{D}^{\times}$ for a quaternion algebra over $F.$  For every essential representation $\overpi$ of $\G(\adele_E) \rtimes \langle \sigma \rangle, \overpi(\sigma)$ admits a preferred factorization $\overpi(\sigma) = \otimes'_v \overpi_v(\sigma_v).$  See our discussion of base change, especially Corollary \ref{localbc}.
\end{rem}

There is a corresponding decomposition of $\sigma$-modules
$$\Hom_{\mathcal{K}}(\wedge^j \widetilde{\mathfrak{p}}, \W[\overpi]^\U \otimes \widetilde{\rho}_\iota) = \Hom_{\G(\adele_E) \rtimes \Galois}(\overpi, \W) \otimes \overpi_{\mathrm{fin}}^\U \otimes \Hom_{\mathcal{K}}(\wedge^j \widetilde{\mathfrak{p}}, \overpi_{\infty} \otimes \widetilde{\rho}_\iota),$$
where $\sigma$ acts trivially on the first factor, by $\sigma_{\mathrm{fin}}$ on the second, and by $\phi \mapsto \sigma_{*} \cdot \phi \cdot (\overpi_{\infty}(\sigma_{\infty}) \otimes \widetilde{\rho}_{\iota}(\sigma_{\infty}))^{-1}$ on the third.  This leads to the identity
\begin{equation} \label{isotypictrace}
\left\langle \Hom_{\mathcal{K}}(\wedge^j \widetilde{\mathfrak{p}}, \W[\overpi]^\U \otimes \widetilde{\rho}_\iota) \right\rangle = m(\overpi) \cdot \left\langle \overpi_{\mathrm{fin}}^\U \right\rangle \cdot \left\langle \Hom_{\mathcal{K}}(\wedge^j \widetilde{\mathfrak{p}}, \overpi_{\infty} \otimes \widetilde{\rho}_\iota) \right\rangle,
\end{equation}
where $m(\overpi) := \dim \Hom_{\G(\adele_E) \rtimes \Galois}(\overpi, \W).$  Using \eqref{isotypictrace}, we can thus express the equivariant zeta function as follows:

\begin{equation} \label{there}
\zeta_{j, \M_\U, V_{\widetilde{\rho}},\sigma}(s) = \sum_{\lambda} \lambda^{-s} \sum_{\overpi \text{ essential}, \lambda_{\widetilde{\rho}_\iota} - \lambda_{\overpi_{\infty}} = \lambda }  m(\overpi) \cdot \left\langle \overpi_{\mathrm{fin}}^\U \right\rangle \cdot \left\langle \Hom_{\mathcal{K}}(\wedge^j \widetilde{\mathfrak{p}}, \overpi_{\infty} \otimes \widetilde{\rho}_\iota) \right\rangle. 
\end{equation}

Specializing \eqref{there} to the case of $E = F$ and the local system over $M = \mathcal{M}$ corresponding to a representation $\rho$ of $\G,$ we arrive at the identity
\begin{equation} \label{untwistedthere}
\zeta_{j, M_U, V_\rho}(s) = \sum_{\lambda} \lambda^{-s} \sum_{\pi, \lambda_{\rho_\iota} - \lambda_{\pi_{\infty}} = \lambda } m(\pi) \cdot \dim \pi_{\mathrm{fin}}^U \cdot \dim \left( \Hom_{K}(\wedge^j \mathfrak{p}, \pi_{\infty} \otimes \rho_\iota) \right). 
\end{equation}

\section[Base change and analytic torsion]{Base change for quaternion algebras and analytic torsion}
\label{quat}

Let $\G$ be the adjoint group of a quaternion algebra over a number field $F.$  Let $E/F$ be a cyclic Galois extension of prime degree $p$ with $\Galois = \langle \sigma \rangle.$  Let $\M = \G(E) \backslash \G(\mathbb{A}_E) / \K \U$ for some Galois-stable maximal compact subgroup $\K \subset \G(E_{\mathbb{R}})$ and a Galois-stable compact open subgroup $ \U \subset \G(\mathbb{A}_E^{\mathrm{fin}}).$  Let $M = \G(F) \backslash \G(\mathbb{A}_F) / KU$ for some maximal compact $K \subset \G(F_{\mathbb{R}})$ and some compact open $U \subset \G(\mathbb{A}_F^{\mathrm{fin}}).$   \smallskip


One main goal of this paper is to prove a spectral comparison of the flavor
\begin{equation} \label{torsionhope}
\log \tau_{\sigma}(\M, \loc) = p \log \tau(M,L).  
\end{equation}
We will prove such spectral identities using trace formula techniques, by comparing the trace of a twisted convolution operator on $L^2(\G(E) \backslash \G(\mathbb{A}_E))$ to that of an untwisted convolution operator on $L^2(\G(F) \backslash \G(\mathbb{A}_F)),$ exactly in the spirit of Langlands' book \cite{Lan} on base change for $\GL_2.$

\begin{itemize}
\item
In $\S \ref{prelimbc},$ we will discuss those elements of the theory of base change need to prove the comparison (\ref{torsionhope})

\item
In $\S \ref{localglobal},$ we discuss local-global compatibility of base change and a consequence for ``multiplicities". 

\item
In $\S \ref{compatibleloc},$ we discuss matching representations and matching local systems.  Such matching pairs occur in the abstract matching theorem proved in $\S \ref{abstractmatching}.$

\item
In $\S \ref{abstractmatching},$ we will prove an abstract matching theorem.  Later, in $\S \ref{unrammatching}, \S \ref{tamematching},$ instances of test functions satisfying the hypotheses of the abstract matching theorem will be described.   
\end{itemize}

\subsection{Preliminaries on base change}
\label{prelimbc}

For a general reference on base change for $\GL_2$ see \cite{Lan}. \medskip

Let $D$ be a quaternion algebra over a number field $F.$  Let $E/F$ be a cyclic Galois extension of degree $p$ with Galois group $\Galois = \langle \sigma \rangle.$  Let $\G$ denote the adjoint group of the group of units of $D.$  In particular, $\G(F) \backslash \G(\mathbb{A}_F)$ is compact.  Assume also that $D_E$ is not globally split, so that $\G(E) \backslash \G(\mathbb{A}_E)$ is compact as well. 


\subsubsection{Twisted convolution and convolution}

\begin{defn}
Let $\mathrm{SM}_c(\G(\adele_E))$ denote the space of \emph{smooth compactly supported measures on $\G(\adele_E)$}.  More precisely, these are all measures of the form $\tilde{f} d\tilde{g}$ for a Haar measure $d\tilde{g}$ on $\G(\adele_E)$ and a function $f = \prod_v f_v.$  We require that $f_{\infty}$ be smooth and compactly supported, that $f_v$ be compactly supported and locally constant, and that $f_v = \mathbf{1}_{\G(O_{E_v})}$ for almost all places $v$ of $F.$  We define $\mathrm{SM}_c(\G(\adele_F))$ similarly. 
\end{defn}

For any smooth, compactly supported measure $\tilde{f} d\tilde{g}$ on $\G(\mathbb{A}_E),$ let 
$$\mathcal{R}(\tilde{f} d \tilde{g}) = \int_{\G(\mathbb{A}_E)} \tilde{f}(g) \mathcal{R}(g) d\tilde{g} \circlearrowleft L^2(\G(E) \backslash \G(\adele_E)),$$
where $\mathcal{R}$ denotes the right regular representation of $\G(\mathbb{A}_E)$ acting on $\G(E) \backslash \G(\mathbb{A}_E).$  Note that the larger group $\G(\mathbb{A}_E) \rtimes \Galois$ acts on $\G(E) \backslash \G(\mathbb{A}_E).$   \bigskip

Similarly, for a smooth, compactly supported measure $f dg$ on $\G(\mathbb{A}_F),$ we let
$$r(f dg) = \int_{\G(\mathbb{A}_F)} f(g) r(g) dg \circlearrowleft L^2(\G(F) \backslash \G(\adele_F)),$$
where $r$ denotes the right regular representation of $\G(\mathbb{A}_F).$

\subsubsection{Matching orbital integrals}
See \cite[$\S 6$]{Lan} for a reference on orbital integrals in the context of base change. \bigskip

Suppose that $E_v = E \otimes_F F_v$ is unramified, i.e. either isomorphic to an unramified cyclic extension of degree $p$ or to $F_v^p.$  We define the norm map
\begin{eqnarray*}
&{}& N : \G(E_v) \rightarrow \G(E_v) \\
&{}& x \mapsto x \sigma(x) ... \sigma^{p-1}(x).
\end{eqnarray*}
If $u = N(x),$ then 
$$\sigma(u) = x^{-1} u x.$$
This implies that the eigenvalues of $u = N(x)$ are preserved by $\Galois.$  Thus, $N(x)$ is $\G(E_v)$-conjugate to an element of $\G(F_v).$  Furthermore, note that
$$N(g^{-1} x \sigma(g)) = g^{-1} N(x) g.$$
So in fact,
$$\{ \sigma \text{-twisted conjugacy classes in } \G(E_v) \} \rightarrow \{ \text{conjugacy classes in } \G(F_v) \}$$
$$[x] \mapsto [N(x)] \cap \G(F_v)$$
is well-defined and injective (see \cite[Lemma 4.2]{Lan}).  For a twisted conjugacy class $c,$ we let $N(c)$ denote the associated conjugacy class in $\G(F_v).$  \\

For a $\sigma$-twisted conjugacy class $c = [x]$ of $\G(E_v)$ and conjugacy class $c' = [x']$ of $\G(F_v),$ we let 

$$\mathcal{O}_{\sigma,c}(\tilde{f}d \tilde{g}) = \int_{\G_{\sigma \ltimes x}(F_v) \backslash \G(E_v)} \tilde{f}( g^{-1} x \sigma(g)) \frac{d \tilde{g}}{dz_{\sigma}},$$

$$\mathcal{O}_{c'}(f dg) = \int_{\G_{x'}(F_v) \backslash \G(F_v)} f(g^{-1} x' g) \frac{dg}{dz},$$

where $\G_{\sigma \ltimes x}$ denotes the twisted centralizer of $x$ and $\G_{x'}$ denotes the centralizer of $x'.$  These definitions do not depend on the choices of representatives $x$ and $x'.$

\begin{defn} \label{matching}
We say that $\tilde{f} d\tilde{g}$ and $f dg$ \emph{match} if 
$$\mathcal{O}_{c'}(f dg) = \begin{cases}
\mathcal{O}_{\sigma,c}(\tilde{f} d \tilde{g}) & \text{ if } c \text{ is a regular twisted conjugacy class and } c' = Nc \\
0 & \text{ if } c' \neq Nc \text{ for any twisted conjugacy class } c. 
\end{cases}$$
Similarly, we say that $\tilde{f} d\tilde{g} = \prod_v \tilde{f}_v d\tilde{g}_v \in \mathrm{SM}_c( \G(\mathbb{A}_E))$ and $f dg = \prod_v f_v dg_v \in \mathrm{SM}_c( \G(\mathbb{A}_F))$  match if they match everywhere locally.
\end{defn}

\begin{rem}
The definitions of $\mathcal{O}_{\sigma,\delta}$ and $\mathcal{O}_{\gamma}$ depend on chocies of Haar measures $dz_{\sigma}$ on $\G_{\sigma \ltimes \delta}(F)$ and $dz$ on $\G_{\gamma}(F).$  However, if $\gamma = N\delta,$ then the group $\G_{\gamma}$ is an inner-twisting of the group $\G_{\sigma\ltimes \delta}.$  Therefore, a choice of Haar measure $dz_{\sigma}$ on $\G_{\sigma \ltimes\delta}(F)$ determines a compatible Haar measure $dz$ on $\G_{\gamma}(F)$ (see \cite[Lemma 5.8]{K1}).  The equality from Definition \ref{matching} is to be taken with respect to this compatible choice.  Though the individual orbital integrals depend on choices of Haar measure, this convention ensures that the notion of $\tilde{f} d\tilde{g}$ and $f dg$ matching is independent of all choices.  
\end{rem}


\begin{exam}[\cite{Lan}, $\S 8$] 
Say that $E_v / F_v$ is split, i.e. $E_v = F_v^p.$  Then for any test function $f_1 dg_1 \times ... \times f_p dg_p$ on $\G(E_v) = \G(F_v)^p,$ the convolution product $(f_1 dg_1) * ... * (f_p dg_p)$ on $\G(F_v)$ matches. 
\end{exam}

\subsubsection{Statement of base change}

%
%
\begin{thm}[Saito, Shintani, Langlands] \label{unrefinedbc}
Suppose that $\tilde{f} = \prod_v \tilde{f}_v d \tilde{g}_v \in \mathrm{SM}_c( \G(\mathbb{A}_E))$ and $f = \prod f_v dg_v \in \mathrm{SM}_c( \G(\mathbb{A}_F))$ are matching test functions.  Then

\begin{equation} \label{bcidentity}
\tr \{ \mathcal{R}(\sigma) \mathcal{R}(\tilde{f} d\tilde{g}) | \W \} = \tr \{ r(f dg) | W \}.
\end{equation}
\end{thm}

Of course, the onus is on us to produce interesting examples of matching test functions.  The identity $(\ref{bcidentity})$ can be leveraged to give a more refined identity.

\begin{cor}[\cite{JL}] \label{refinedbc}
Let $E/F$ be a cyclic Galois extension of number fields of \emph{odd} prime degree.  There is a bijection between irreducible automorphic representations $\pi$ of $\G$ and essential (see definition \ref{essential}) automorphic representations $\overpi$ of $R_{E/F} \G_E \rtimes \Galois.$  This bijection is uniquely defined by an equality of traces: $\pi$ and $\overpi$ match if and only if
\begin{equation} \label{globalisotypic}
\tr \{ \mathcal{R}(\sigma) \mathcal{R}(\tilde{f} d\tilde{g}) | \W[\overpi] \} = \tr \{ r(f dg) | W[\pi] \}
\end{equation}
for all pairs of matching test functions $\tilde{f} d\tilde{g}$ and $f dg$ as above.  Here, $\W[\overpi]$ denotes the $\overpi$ isotypic subspace of  $\W$ and $W[\pi]$ the $\pi$-isotypic subspace of $W.$  
\end{cor}

\subsection{Local-global compatibility for base change and a consequence for multiplicities}
\label{localglobal}

The groups $R_{E/F}\G_E$ and $\G$ satisfy multiplicity one, by the Jacquet-Langlands correspondence and multiplicity one for $\GL_2.$  Therefore, the equality of equation (\ref{globalisotypic}) is equivalent to
\begin{equation} \label{singlerep}
\tr \{ \overpi(\sigma) \overpi(\tilde{f} d\tilde{g}) \} = \tr \{ \pi(f dg) \}
\end{equation}
for all pairs of matching test functions $\tilde{f} d\tilde{g}$ and $f dg.$  Using equation (\ref{singlerep}), Langlands proves a local version of base change and local-global compatiblity.

\begin{thm}[\cite{Lan}, $\S 7$] \label{localbc}
Let $E/F$ be a cyclic Galois extension of odd prime degree $p.$  There is a bijection between irreducible representations $\pi_v$ of $\G(F_v)$ and representations $\overpi_v$ of $\G(E_v) \rtimes \Galois$ which are irreducible after restricting to $\G(E_v).$  This bijection is determined by the equality of traces 
$$\tr \{ \overpi_v(\sigma_v) \overpi_v(\tilde{f}_v \tilde{g}_v) \} = \tr \{ \pi_v(f_v dg_v) \}$$
for all pairs of matching test functions $\tilde{f}_v d\tilde{g}_v$ and $f_v dg_v.$  This bijection is compatible with the global bijection of Theorem \ref{refinedbc} in the sense that $\pi = \otimes' \pi_v$ globally matches the unique essential representation $\overpi = \otimes' \overpi_v.$  
\end{thm}

\begin{rem}
As promised in Remark \ref{factorization}, the above local-global compatibility theorem has the following consequence: for any essential representation $\overpi = \otimes' \overpi_v$ of $\G(\adele_E) \rtimes \Galois$ matching the representation $\pi = \otimes' \pi_v$ of $\G(\adele_F),$ there is a canonical factorization
$$\overpi(\sigma) = \otimes' \overpi(\sigma_v)$$  
\end{rem}

\subsubsection{Local base change and ``trace matching"}
\label{tracematching}

\begin{defn} \label{tracematchingtestfunctions}
Let $\tilde{f} d\tilde{g}$ be a smooth measure on $\G(E_v)$ and let $f dg$ be a smooth measure on $\G(F_v).$  We say that $\tilde{f} d\tilde{g}$ and $f dg$ are \emph{trace-matching} if for every irreducible admissible representation $\pi$ of $\G(F_v),$ there is an equality
$$\tr \{ BC(\pi)(\sigma) BC(\pi)(\tilde{f} d\tilde{g})\}  = \tr \{ \pi(fdg) \}.$$ 
\end{defn}   

\begin{exam}
If $\tilde{f} d\tilde{g}$ and $f dg$ have matching orbital integrals in the sense of definition \ref{matching}, then they are trace-matching.  
\end{exam}

We record one straightforward consequence that Theorem \ref{localbc} on local-global compatibility of base change has for comparing ``multiplicities":

\begin{lem} \label{multiplicitymatching}
Let $S$ be any set of places.  Let $\U = \prod_{v \in S} \U_v$ be a compact open subgroup of the restricted tensor product $\prod'_{v \in S} \G(E_v).$  Suppose that $\mathbf{1}_{\U_v}du_v$ trace-matches $m_v$ for each $v \in S$ and that $\prod_{v \in S} m_v = \sum c_U \mathbf{1}_U du.$  Then for matching representations $\overpi = \otimes'_v \overpi_v$ of $\G(\adele_E) \rtimes \Galois$ and $\pi = \otimes \pi_v$ of $\G(\adele_F),$ there is an equality
$$\left\langle \overpi_S^\U \right\rangle = \sum_U c_U \dim \pi_S^U,$$
where $\bullet_S$ denotes the (restricted) tensor or cartesian product over all places $v \in S.$
\end{lem}

\begin{proof}
This follows directly by theorem \ref{localbc} and definition \ref{tracematching} of trace-matching:
\begin{eqnarray*}
\left\langle \overpi_S^\U \right\rangle &:=& \tr \{  \overpi_S(\sigma) \overpi_S(\mathbf{1}_\U du)  \} \\
&=& \prod_{v \in S} \tr \{ \overpi_v(\sigma) \overpi_v(\mathbf{1}_{\U_v} d\tilde{u}_v) \} \\
&=& \prod_{v \in S} \tr \{ \pi( m_v ) \} \\
&=& \tr \{ \pi_S \left( \sum  c_U \mathbf{1}_U du \right) \} \\
&=& \sum_U c_U \dim \pi_S^U.
\end{eqnarray*}
\end{proof}

\subsection{Local systems and matching representations}
\label{compatibleloc}

Let $\G$ be a semisimple group over the number field $F.$  Let $i_E: F \hookrightarrow E$ be a cyclic Galois extension of prime degree $p.$  Note that $(R_{E/F} \G)_E = \prod_{\Galois} \G_E,$ where $\Galois$ acts by permuting the factors according to its left translation action on itself.  Let $i_N: F \hookrightarrow N$ be a second finite extension.  Let $R_{E/F} \G \rightarrow R_{N/F} \GL(V),$ for $V$ a finite dimensional $N$-vector space, be an algebraic homomorphism defined over $F.$
%
%

\subsubsection{Definition and examples of matching representations and local systems}
Fix a complex embedding $\iota: N \hookrightarrow \mathbb{C}.$  Let $\C_{\iota} = N \otimes_{\iota} \C.$  Then $R_{E/F} \G_E(\C_{\iota}) = \prod_{\iota'} \G(\C_{\iota'})$ where the product runs over all $\iota': E \hookrightarrow \mathbb{C}$ extending $\iota \circ i_N.$  The Galois group $\Galois$ acts on this product by permuting the factors according to its action on the set $\{ \iota'|\iota \}.$  Furthermore, the induced homomorphism $R_{E/F} \G(F \otimes \R) \rightarrow \GL(V)(\C_{\iota})$ factors as 
$$R_{E/F} \G(F \otimes \R) \rightarrow R_{E/F} \G(F \otimes \C) \rightarrow R_{E/F} \G(\C_{\iota}) \rightarrow \GL(V)(\C_{\iota}),$$
the second map being induced by the projection $F \otimes \C \rightarrow \C_{\iota}.$

%
%
%
%

\begin{defn}[matching representations for the same coefficient field] \label{compatible}
Let $\widetilde{\rho}: R_{E/F} \G_E \rtimes \Galois \rightarrow R_{N/F} \GL(V_{\widetilde{\rho}})$ and $\rho: \G \rightarrow R_{N/F} \GL(V_{\rho})$ be representations of algebraic groups over $F.$  We say that $\widetilde{\rho}$ and $\rho$ \emph{match} if
$$\widetilde{\rho}_{\iota} = \otimes_{\iota' | \iota} \rho_{\iota},$$
with $\Galois$ acting by permutations on the set $\{ \iota' | \iota \}.$  For local systems $V_{\widetilde{\rho}}$ and $V_{\rho}$ of $N$-vector spaces arising in the manner of $\S \ref{rational},$ we say that $V_{\widetilde{\rho}}$ and $V_{\rho}$ match exactly when $\widetilde{\rho}$ and $\rho$ match. 
\end{defn}

\begin{defn}[matching representations for different coefficient fields] \label{compatible'}
Let $\widetilde{\rho}: R_{E/F} \G_E \rtimes \Galois \rightarrow R_{E/F} \GL(V_{\widetilde{\rho}})$ and $\rho: \G \rightarrow \GL(V_{\rho})$ be representations of algebraic groups over $F.$  We say that $\widetilde{\rho}$ and $\rho$ \emph{match} if
$$\widetilde{\rho}_{\iota} = \otimes_{\iota' | \iota} \rho_{\iota},$$
with $\Galois$ acting by permutations on the set $\{ \iota' | \iota \}.$  For local systems $V_{\widetilde{\rho}}$ of $E$-vector spaces and $V_{\rho}$ of $F$-vector spaces arising in the manner of $\S \ref{rational},$ we say that $V_{\widetilde{\rho}}$ and $V_{\rho}$ match exactly when $\widetilde{\rho}$ and $\rho$ match. 
\end{defn}

\textbf{Main examples of matching representations.}

\begin{itemize}
\item
Let $\rho: \G \rightarrow R_{N/F} \GL(V)$ be any representation over $F.$  Suppose that $N$ is a Galois closure of $F$ containing $E.$   For any $F$-algebra $A, \rho$ induces a homomorphism 
$$\G(A \otimes_F E) \rightarrow GL((V \otimes_F A) \otimes_F E \otimes_F N) \rightarrow \prod_i GL((V \otimes_F A) \otimes_F N) \rightarrow GL \left( \bigotimes_i (V \otimes_F A) \otimes_F N \right)$$  
where the product runs over all $i$ for which $F \xrightarrow{i_E} E \xrightarrow{i} N$ and $i \circ i_E = i_N$ where $i_E$ and $i_N$ are the embeddings defining $E/F$ and $N/F.$  The second map above is induced by the ring map
\begin{eqnarray*}
A \otimes_F E \otimes_F N &\rightarrow& A \otimes_F N \\
a \otimes_F e \otimes_F n &\mapsto& ( a \otimes_F i(e) \cdot n)_i.
\end{eqnarray*}
The above homomorphisms are functorial in $A$ and so define a map 
$$R_{E/F} \G \rightarrow R_{F/N} \GL \left(\bigotimes_i V \right).$$
This map extends naturally to 
$$\widetilde{\rho}: R_{E/F} \G \rtimes \Galois \rightarrow R_{F/N} GL \left(\bigotimes_i V \right),$$
where $\Galois$ acts by permuting the embeddings $i.$  The homomorphisms $\rho$ and $\widetilde{\rho}$ match.

\item
Let $\rho: \G \rightarrow \GL(V)$ be any representation over $F.$  There is a natural map  
$$ \G(A \otimes_F E) \rtimes \Galois \rightarrow GL(V \otimes_F A \otimes_F E)$$
given by composing $\rho$ on $A \otimes_F E$-valued points with the $F$-algebra homomorphism
\begin{eqnarray*}
A \otimes_F E &\rightarrow& A \otimes_F E \\
a \otimes_F e  &\mapsto& a \otimes_F \sigma(e)
\end{eqnarray*}
for each $\sigma \in \Galois.$  This collection of maps is functorial in $F$-algebras $A$ and so defines an algebraic group homomorphism over $F$
$$\widetilde{\rho}: R_{E/F} \G \rtimes \Galois \rightarrow R_{E/F}\GL (V_{\rho}).$$
The representations $\rho$ and $\widetilde{\rho}$ match. 
\end{itemize}

\subsection{An abstract matching theorem}
\label{abstractmatching}

\subsubsection{Notational setup for the matching theorem}

Recall that $\mathcal{M}_{\mathcal{U}} = \G(E) \backslash \G(\mathbb{A}_E) / \mathcal{K} \mathcal{U},$ where $\mathcal{K}, \mathcal{U}$ are chosen to be $\Galois$-stable.  This locally symmetric space is then acted on by $\Galois = \langle \sigma \rangle,$ which is cyclic of prime degree $p,$ by isometries. Similarly, for compact open $U \subset \G(\finadele_F)$ and maximal compact $K,$ we let $M_U = \G(E) \backslash \G(\adele_F) / KU.$  Let $\rho: \G \rightarrow R_{F/N}\GL(V_{\rho})$ and $\widetilde{\rho}: R_{E/F} \G \rightarrow R_{F/N} \GL(V_{\widetilde{\rho}})$ be matching representations and let $\iota: N \hookrightarrow \C$ be a complex embedding.  We recall the notation
$$\W_{\widetilde{\rho}} = L^2(\G(E) \backslash \G(\adele_E) ) \otimes \widetilde{\rho}_\iota,$$
$$\W_{\U, \widetilde{\rho}} = L^2(\G(E) \backslash \G(\adele_E) / \U) \otimes \widetilde{\rho}_\iota,$$
$$W_{\rho} = L^2(\G(F) \backslash \G(\adele_F) ) \otimes \rho_\iota,$$
$$W_{U, \rho} = L^2(\G(E) \backslash \G(\adele_E) / U) \otimes \rho_\iota.$$
For any compact group $J,$ we let $dj$ denote its volume 1 Haar measure.

\subsubsection{Statement and proof of the matching theorem}

The goal of this section is to prove the following theorem:

\begin{thm} \label{abstractmatchingthm}
Assume that that
$$\prod \mathbf{1}_{\U_v} du_v = \mathbf{1}_{\mathcal{U}} d\tilde{u} \text{ and } \prod m_v = \sum_{\Pi(\U)} c_U \mathbf{1}_U du \text{ are trace-matching test functions}$$
 (see Definition \ref{tracematchingtestfunctions}), where $\Pi(\U)$ is a finite set of compact open subgroups of $\G(\finadele_F).$  Let $\widetilde{\rho}$ and $\rho$ be matching representations of $R_{E/F} \G_E \rtimes \Galois$ and $\G$ (defined in $\S \ref{compatibleloc}$).  Then
\begin{equation}
\zeta_{j,\M_{\U}, V_{\widetilde{\rho}}, \sigma}(s) = \begin{cases}
p (-1)^{a^2(p-1)} \sum_{\Pi(U)} c_U \; \zeta_{a,M_U, V_{\rho}}(s) \cdot p^{-s} & \text{ if } j = pa \\
0 & \text{ if } j \neq 0 \mod p.
\end{cases}
\end{equation}
\end{thm}

\begin{proof}
Because the groups $\G$ and $R_{E/F} \G$ satisfy multiplicity one, we can rewrite the expression derived in $\S \ref{reptorsion}$ for the above zeta functions from equations (\ref{there}) and (\ref{untwistedthere}) as: 
\begin{equation} \label{upstairs}
\zeta_{j, \M_\U, V_{\widetilde{\rho}},\sigma}(s) = \sum_{\lambda} \lambda^{-s} \sum_{\overpi \text{ essential}, \lambda_{\widetilde{\rho}} - \lambda_{\overpi_{\infty}} = \lambda }  \left\langle \overpi_{\mathrm{fin}}^\U \right\rangle \cdot \left\langle \Hom_{\mathcal{K}}(\wedge^j \widetilde{\mathfrak{p}}, \overpi_{\infty} \otimes \widetilde{\rho}_\iota) \right\rangle. 
\end{equation}
\begin{equation} \label{downstairs}
\zeta_{j, M_U, V_\rho}(s) = \sum_{\lambda} \lambda^{-s} \sum_{\pi, \lambda_{\rho} - \lambda_{\pi_{\infty}} = \lambda } \dim \pi_{\mathrm{fin}}^U \cdot \dim \left( \Hom_{K}(\wedge^j \mathfrak{p}, \pi_{\infty} \otimes \rho_\iota) \right). 
\end{equation}


We now discuss matching for the multiplicities, the Casimir eigenvalues, and the $(\mathfrak{g},K)$-cochain groups in turn.

\begin{itemize}
\item
By Lemma \ref{multiplicitymatching}, for any matching pair of an essential representation $\overpi$ of $\G(\adele_E) \rtimes \Galois$ and an irreducible representation $\pi$ of $\G(\adele_F),$ there is an equality 
\begin{equation} \label{multiplicities}
\left\langle \overpi_{\mathrm{fin}}^\U \right\rangle = \sum_U c_U \dim \pi_{\mathrm{fin}}^U.
\end{equation}
\item
We can relate the Casimir eigenvalues of $\overpi_{\infty}$ and $\pi_{\infty}$ by being more explicit about the relationship between $\overpi_{\infty}$ and $\pi_{\infty}.$  Because $[E:F]$ is odd $E_\R \cong (F_\R)^p.$  A choice of isomorphism determines an isomorphism $\G(E_\R) = \G(F_\R)^p.$  Langlands calculates, in \cite[$\S 8$]{Lan}, that the representation $\pi_{\infty}$ corresponds to the representation $\overpi_{\infty} = \pi_{\infty}^{\boxtimes p}$ of $\G(E_\R) \rtimes \Galois$ equipped with intertwining isomorphism given by a cyclic shift. Similarly, $\widetilde{\rho}_\iota$ is isomorphic to $\rho^{\boxtimes p}_\iota$ with intertwining isomorphism given by the same cyclic shift.  Therefore, 
\begin{equation} \label{casimireigenvalues}
\lambda_{\widetilde{\rho}_\iota} - \lambda_{\overpi_{\infty}} = p ( \lambda_{\rho_\iota} - \lambda_{\pi_{\infty}}).
\end{equation} 

\item
Finally, note that because $\overpi_{\infty} \otimes \widetilde{\rho}_{\iota} \cong \pi_{\infty}^{\boxtimes p} \otimes \rho_{\iota}^{\boxtimes p}$ the $(\mathfrak{g},K)$-chain complexes obey a ``K\"{u}nneth relationship": 
$$\Hom_{\K}(\wedge^{\bullet} \widetilde{\mathfrak{p}}, \overpi_{\infty} \otimes \widetilde{\rho}_\iota) = \Hom_K(\wedge^\bullet \mathfrak{p}, \pi_{\infty} \otimes \rho_\iota)^{\otimes p} \hspace{0.5cm} \text{ (graded tensor product)}$$ 
(see \cite[I.1.3]{BW}) which is equivariant for the cyclic shift on the right and the $\sigma$ action on the left.  Exactly as in our calculations from \cite[
Proposition 5.1]{Lip1}, it follows that
\begin{equation} \label{gK}
\langle\Hom_\K(\wedge^j \widetilde{\mathfrak{p}}, \overpi_{\infty} \otimes \widetilde{\rho}_\iota)\rangle = \begin{cases}
(-1)^{a^2(p-1)} \cdot \dim \left( \Hom_K(\wedge^a \mathfrak{p}, \pi_{\infty} \otimes \rho_\iota) \right) & \text{ if } j = pa \\
0 & \text{ if } j \text{ is not a multiple of } p. 
\end{cases}
\end{equation}  
\end{itemize}

By Langlands' formulation of base change (see Theorem \ref{refinedbc}), there is a bijection between irreducible subrepresentations $\pi \subset W$ of $\G(\adele_F)$ and essential representations $\overpi \subset \W$ of $\G(\adele_E) \rtimes \Galois.$  Multiplying $(\ref{multiplicities}), (\ref{casimireigenvalues})^{-s}$ and $(\ref{gK})$ together and summing over all subrepresentations $\pi$ of $W$ therefore gives the result. 
\end{proof}

An immediate corollary concerning analytic torsion is as follows:

\begin{cor} \label{abstracttorsion}
$$\log \tau_{\sigma}(\M_\U, V_{\widetilde{\rho}}) = p \left( \sum c_U  \log \tau(M_U,V_{\rho}) - \log(p) Z_{M_U,V_{\rho}}(0) \right).$$
\end{cor}

\begin{rem} 
The second summand on the right hand side is a red herring.  It would disappear by scaling the metric on $M_U$ by $p.$  So in particular, if $V_{\rho, \mathbb{Q}}$ is acyclic, the analytic torsion is metric independent and so the second summand automatically vanishes.  
\end{rem}

\section{Matching at places where $E/F$ is unramified} 
\label{unrammatching}

We recall our notational setup:
Let $\G$ be the adjoint group of the group of units of a quaternion algebra $D$ over a number field $F.$  Say $E/F$ is a cyclic Galois extension of prime degree $p.$  The Galois group $\Galois = \langle \sigma \rangle$  acts on $\G(\mathbb{A}_E).$ \bigskip

%
%

The goal of this section is to describe instances of local trace-matching (see Definition  \ref{tracematching}) at places $v$ for which $E_v / F_v$ unramified (see, in particular, Theorem \ref{parahoricmatching}); this will enable us to prove relationships between twisted analytic torsion on locally symmetric spaces related by base change - cf. Theorem \ref{abstractmatching}.

\subsection{Parahoric level structure}
Throughout this section, assume that $E/F$ is everywhere unramified. \bigskip

The argument for proving Theorem \ref{abstractmatching} relating $\tau_{\sigma}(V_{\widetilde{\rho}})$ to $\tau(V_{\rho})$ hinges on the fact that  $\mathbf{1}_\U d \tilde{u}$ and $m = \mathbf{1}_{U^{\Sigma}} du^{\Sigma} \times \prod_{v \in \Sigma} m_v$ and are matching test functions.  We introduce a type of level structure for which we will be able to prove such a matching theorem.

\begin{defn}[Parahoric level structure] \label{parahoric}
Call a compact open subgroup $\U_v \subset \G(E_v)$ \emph{locally parahoric} if it satisfies the following conditions: 

\begin{itemize}
\item
If $D$ is ramified at $v,$ then $\U_v$ should equal the image of the units of the maximal order of $D_v$ in $\G(E_v).$  

\item
Suppose $E_v = E \otimes_F F_v$ is split and $D$ is unramified at $v.$   Identify $\G(E_v) = \G(F_v)^p.$  Then $\U_v = U_v^p$ for an arbitrary compact open $U_v \subset \G(F_v).$   

\item
Say $E_v / F_v$ is unramified, cyclic of degree $p,$ and that $D$ is unramified at $v.$  The tree of $\G / F_v$ injects into the tree of $\G / E_v$ with image identified as the Galois invariants.  We insist that $\U_v = J,$ where $J$ is the pointwise stabilizer in $\G(E_v)$ of either a vertex or an edge of the tree of $\G / F_v.$
\end{itemize}

If $\U = \prod \U_v$ is locally parahoric at all places $v$ of $F,$ then we call $\U$ \emph{(globally) parahoric}.  If $\U$ is locally parahoric for all $v \notin \Sigma,$ then we call $\U$ \emph{parahoric outside} $\Sigma.$ 
\end{defn}

\begin{rem}
The above definition does not apply at places $v$ where $E_v / F_v$ is ramified.  We extend the definition of parahoric level structure, for places $v$ where $E_v / F_v$ is tamely ramified, in Definition $\ref{tameparahoric}.$
\end{rem}

To every parahoric level structure $\U \subset \G(\mathbb{A}_E^{\mathrm{fin}}),$ we associate a matching level structure $U \subset \G(\mathbb{A}_F^{\mathrm{fin}}).$

\begin{defn} \label{associatedlevel}
Let $\U_v \subset \G(E_v)$ be parahoric at $v.$  We say that $U_v$ is \emph{associated to} $\U_v$ if one of the following hold:

\begin{itemize}
\item
Suppose $D_v$ ramified.  We require that $U_v$ equals the image of the units of the maximal order of $D_v$ in $\G(F_v).$

\item
Suppose that $E_v$ split.  Then $\U_v = (U'_v)^p$ for some $U'_v \subset \G(F_v),$ and we require that $U_v = U'_v.$

\item
Suppose that $E_v / F_v$ is an unramified field extension and that $D_v$ unramified.  Then $\U_v$ equals the pointwise stabilizer in $\G(E_v)$ of some simplex of the tree of $\G / F_v,$ viewed as a subset of the tree of $\G / E_v.$  We require that $U_v$ be the stabilizer of that same simplex in $\G(F_v).$ 
\end{itemize}

We say that $U = \prod U_v$ is associated to $\U = \prod U_v$ if $U_v$ is \emph{associated to} $\U_v$ for all places $v.$  
\end{defn}

With these notions in hand, we now throw the kitchen sink at the issue of matching $\mathbf{1}_{\mathcal{U}} d\tilde{u} \leftrightarrow \mathbf{1}_{U} du.$

\begin{thm}[Kottwitz] \label{parahoricmatching}
Let $U_v$ be associated to a parahoric level structure $\U_v$ at $v.$  Then the test functions $\mathbf{1}_{\mathcal{U}_v} d \tilde{u}_v$ and $\mathbf{1}_{U_v} du_v$ geometrically match, where $d\tilde{u}_v$ and $du_v$ are volume 1 Haar measrues.  In particular,  $\mathbf{1}_{\mathcal{U}_v} d \tilde{u}_v$ and $\mathbf{1}_{U_v} du_v$ are trace-matching (see Definition \ref{tracematchingtestfunctions}). 
\end{thm}

\begin{proof}
In \cite{K}, Kottwitz proves that 
$$SO_{\delta,\sigma}( \mathbf{1}_{\mathcal{U}_v} d \tilde{u}_v) = SO_{N\delta}( \mathbf{1}_{U_v} du_v )$$
for any regular semisimple $\delta \in \G(E_v).$  Here $SO$ denotes the stable orbital integral, whose definition is given in \cite{K}.  Let $c$ be the $\G(F_v)$-conjugacy class of some element of $\G(F_v).$   Suffice it to say that if
$$\{ \G(\overline{F}_v) \text{ conjugacy class of } c \} \cap \G(F_v) = c,$$
for every regular semisimple conjugacy class, then $SO_{\delta,\sigma} = O_{\delta,\sigma}$ and $SO_{\gamma} = O_{\gamma}.$  Fortunately, this is precisely the situation we're in:

\begin{itemize}
\item
Say that $\G'$ is any reductive group over $F_v$ with center $\mathbf{Z}.$  If $\mathbf{Z}$ is split, then $\G'/\mathbf{Z}(F_v) = \G'(F_v) / \mathbf{Z}(F_v).$  Indeed, there is a Galois cohomology exact sequence
$$0 \rightarrow \G'(F_v) / \mathbf{Z}(F_v) \rightarrow  \G'/\mathbf{Z}(F_v) \rightarrow H^1(F_v, \mathbf{Z}).$$
If $\mathbf{Z}$ is split, the rightmost term vanishes, by Hilbert's theorem 90, and so $\G'/\mathbf{Z}(F_v) = \G'(F_v)/ \mathbf{Z}(F_v).$  The center of the unit group of a \emph{central} simple algebra is always split.

\item
Now suppose that $x,x' \in \underline{D}^{\times}(F_v)$ are regular semisimple and conjugate by $\underline{D}^{\times}(\overline{F}_v).$  Then the characteristic polynomials
$$p_x(t) = t^2 - \text{trd}(x) t + \text{nrd}(x), p_{x'}(t) = t^2 - \text{trd}(x')t + \text{nrd}(x')$$
are equal, say to $p(t).$  By the Noether-Skolem theorem, the two embeddings
$$F_v[t]/(p(t)) \xrightarrow{i_x, i_{x'}} D, i_x(t) = x, i_{x'}(t) = x'$$
are conjugate by $D_{F_v}^{\times},$ exactly as we wanted.  Note that this argument works whether or not $D_{F_v}$ is ramified.   
\end{itemize}

The above arguments show that for $\G$ the adjoint group of a the units of a quaternion algebra over $F_v, \G(\overline{F}_v)$-conjugacy collapses to $\G(F_v)$-conjugacy.  Thus, $\mathbf{1}_{\mathcal{U}_v} d \tilde{u}_v$ and $\mathbf{1}_{U_v}du_v$ match, as claimed.
\end{proof}

\section[Matching at tamely ramified places]{Matching at places where $E/F$ is tamely ramified}
\label{tamematching}

To broaden the applicability of Theorem \ref{abstractmatching}, we need to prove matching theorems at places $v$ where $E_v/F_v$ is ramified.  We assume throughout this chapter that $E_v/F_v$ be \emph{tamely ramified}.  Our immediate goals:

\begin{itemize}

\item[(1)]
characterize all (possibly ramified) representations $\pi$ of $\PGL_2(F_v)$ for which $BC(\pi)$ is unramified.  

\item[(2)]
Find a test function on $f dg$ on $\PGL_2(F_v)$ which trace-matches $\mathbf{1}_{\PGL_2(E_v)} d\tilde{g}.$
\end{itemize}

Characterizing those representations as in (1) will be relatively straightforward; this is done in $\S \ref{tamelanglandsparameters}.$  In analogy with the terminology used for elliptic curves, we will say that representations as in (1) have \emph{additive reduction}.  For (2), the main content is an analysis of $J$-fixed vectors in principal series representations for various compact open subgroups $J \subset \PGL_2(O_F)$; this is accomplished in $\S \ref{principalseries}.$  A test function which serves as an indicator function for representations with additive reduction is constructed in Theorem $\ref{additivetestfunction}.$ \bigskip 

\textbf{Notation for local fields}.  To ease notation in what follows, we will denote $E_v/F_v$ by $E/F.$  So for the remainder of this chapter, $E/F$ will denote a tamely ramified, cyclic Galois extension of local fields, \emph{not} an extension of global fields.

\subsection{Tamely ramified Langlands parameters and additive reduction}
\label{tamelanglandsparameters}

We want to characterize continuous representations
$$\sigma: W_F \rightarrow SL_2(\overline{\mathbb{Q}_{\ell}})$$
such that $\sigma|_{W_E}$ is unramified.  Since $E/F$ is tamely ramified, such representations $\sigma$ are tamely ramified.  But the tame Weil group has a very simple description:

$$W_F^{\textrm{tame}} = I_F^{\textrm{tame}} \rtimes \langle F \rangle$$

where $I_F^{\textrm{tame}}$ is procyclic, generated by a single element $u,$ and $F$ is any lift of the arithmetic Frobenius.  The generators $F$ and $u$ satisfy the single relation
$$FuF^{-1} = u^q, \text{ where } q = \# k_F.$$ 
\begin{rem}  
Because $\Galois$ is an abelian quotient of $W_F^{\textrm{tame}}$ of prime order $p,$ it follows that $p | q - 1.$ 
\end{rem}

We now prove a lemma characterizing all tame Langlands parameters of additive reduction.

\begin{lem}
All Langlands parameters of additive reduction are either unramified or conjugate to
$$\sigma = \sigma_{a,\zeta}:= F \mapsto \left( \begin{array}{cc}
a & 0 \\
0 & a^{-1}
\end{array}\right), u \mapsto \left( \begin{array}{cc}
\zeta & 0 \\
0 & \zeta^{-1}
\end{array} \right).$$
\end{lem}

\begin{proof}
Note that $I_F^{\textrm{tame}} / I_E^{\textrm{tame}} = \Galois,$ which is a cyclic group of order $p.$  Thus, $I_E^{\textrm{tame}}$ must be the unique closed subgroup of index $p$ in $I_F^{\textrm{tame}}.$  In particular, $u^p \in I_E^{\textrm{tame}}.$  It follows that if $\sigma(F) = A, \sigma(u) = B,$ then 
$$ABA^{-1} = B^q \text{ and } B^p = 1.$$
After conjugation if necessary, we may suppose that $B = \left( \begin{array}{cc}
\zeta & 0 \\
0 & \zeta^{-1} \end{array} \right),$ where $\zeta^p = 1.$  

\begin{itemize}
\item
Assume that $\zeta \neq 1.$  Let $A = \left( \begin{array}{cc} 
a & b \\
c & d
\end{array}\right).$  Using the fact that $ABA^{-1} = B^q$ and that $\zeta^{q-1} = 1,$ we see that $b = c = 0.$  These give rise to the Langlands parameters
$$\sigma = \sigma_{a,\zeta}:= F \mapsto \left( \begin{array}{cc}
a & 0 \\
0 & a^{-1}
\end{array}\right), u \mapsto \left( \begin{array}{cc}
\zeta & 0 \\
0 & \zeta^{-1}
\end{array} \right).$$
These parameters correspond, by the local Langlands correspondence, to the representations $I_{s,\chi}:= \mathrm{Ind}_B^G ( |\cdot|^s \cdot \chi)$ where $\chi$ is a non-trivial character of $F^{\times} / N(E^{\times}).$

\item
If $\zeta = 1,$ then the corresponding parameter corresponds to an unramified representation.
\end{itemize}
\end{proof}


\subsection{A closer look at local base change}

Let $E/F$ be a tamely ramified extension of local fields.  The goal of this section is to construct a smooth, compactly supported test measure $m_{\mathrm{add}}$ on $\PGL_2(F)$ such that $\mathbf{1}_{\PGL_2(O_E)}d\tilde{k}$ trace-matches $m_{\mathrm{add}}$  (see Definition \ref{tracematchingtestfunctions}).  

\begin{itemize}
\item
In $\S \ref{orbitalintegralsviatraces},$ we explain how to recover the matching of orbital integrals in a more representation theoretic manner.  

\item
In $\S \ref{twistedspherical},$ we discuss the action of $\Galois$ on the $\PGL_2(O_E)$-fixed vectors of a representation of $\PGL_2(E)$ of the form $BC(\pi).$
\end{itemize}

\subsubsection{Strategy for proving the local base change trace identity}
\label{orbitalintegralsviatraces}

Recall the following definition from $\S \ref{quat}$ on base change:

\begin{defn}
A representation $\overpi$ of $\PGL_2(E) \rtimes \langle \sigma \rangle$ \emph{matches} the representation $\pi$ of $\PGL_2(F)$ if the following equality of traces holds:
\begin{equation} \label{localmatchingformula}
\tr \{ \overpi(\sigma) \overpi(\tilde{f} d\tilde{g}) \} = \tr \{ \pi(f dg)  \} 
\end{equation}
for every pair of matching test function $\tilde{f} d\tilde{g} \in \mathrm{SM}_c(\G(E)), f dg \in \mathrm{SM}_c(\G(F)).$  We write $\overpi = BC(\pi).$
\end{defn}

\begin{rem}
The intertwining isomorphism $\overpi(\sigma): \pi \rightarrow \pi^{\sigma}$ is a priori only well-defined up to a $p$th root of unity.  However, the equality of traces in the above definition uniquely determines $\overpi(\sigma).$
\end{rem}

%
%
%
%
%

For the particular choices of test functions $\tilde{f}_0 d\tilde{g}_0$ and $f_0 dg_0$ that we will make in the sequel, rather than proving that equation (\ref{localmatchingformula}) holds by demonstrating an equality of orbital integrals, we prove it directly by using explicitly understood features of local base change for $\PGL_2(F).$

\subsubsection{Action of $\sigma$ on $\PGL_2(O_E)$-fixed vectors}
\label{twistedspherical}

Let $\mu$ be the inflation to the upper triangular Borel subgroup $B \subset \PGL_2(F)$ of a character $T \rightarrow \mathbb{C}^{\times},$ where $T$ is the diagonal torus of $\PGL_2(F)$ and let $\delta$ denote the modulus character of $B.$  We let $(\rho_{\mu}, I_{\mu})$ denote the normalized principal series representation, i.e. the space of locally constant functions $f: \PGL_2(F) \rightarrow \mathbb{C}$ which transform by the rule $f(bg) = \mu(b) \delta(b)^{1/2} f(g)$ where $\rho_{\mu}$ acts by right translation.

We define $I_{\widetilde{\mu}}$ similarly for characters of the upper triangular Borel subgroup $\widetilde{B} \subset \PGL_2(E).$  If $\widetilde{\mu}^{\sigma} = \widetilde{\mu},$ then we extend $I_{\widetilde{\mu}}$ to a representation of $\PGL_2(E) \rtimes \Galois$ by the formula
$$\rho_{\widetilde{\mu}}(\sigma) f(g) = f(g^{\sigma^{-1}}).$$ 
\begin{prop}[\cite{Lan}, Corollary 7.3] \label{sphericalaction}
Let $N: E^{\times} \rightarrow F^{\times}$ denote the norm map.  The representations $I_{\mu \circ N}$ and $I_{\mu}$ match.  That is, for any pair $\tilde{f} d\tilde{g}$ and $f dg$ of matching test functions on $\PGL_2(E)$ and $\PGL_2(F)$ respectively, there is an equality

\begin{equation*}
\tr \{ \rho_{\mu \circ N}(\sigma) \rho( \tilde{f} d\tilde{g} )  \} = \tr \{ \rho_{\mu}(f dg) \}.
\end{equation*}
\end{prop}

\begin{cor} \label{nodifference}
For every irreducible admissible representation $\pi$ of $\PGL_2(F),$ there is an equality
\begin{eqnarray*}
&{}& \tr \{ BC(\pi)(\sigma) BC(\pi)(\mathbf{1}_{\PGL_2(O_E)} d\tilde{k}) \} \\
&=& \tr \{ BC(\pi)(\mathbf{1}_{\PGL_2(O_E)} d\tilde{k})  \} = \begin{cases}
1 & \text{ if } BC(\pi) \text{ is unramified } \\
0 & \text{ otherwise.}
\end{cases}
\end{eqnarray*}
\end{cor}

\begin{proof}
The claim is vacuously true if $BC(\pi)$ is ramified.  If it is unramified, then it can be embedded in an unramified principal series representation $V$; indeed, any representation with an Iwahori fixed vector can be embedded into an unramified principal series representation.  

\begin{itemize}
\item
If $V$ is irreducible, then $BC(\pi) = V.$  But Proposition $\ref{sphericalaction}$ shows that for such principal series representations, in the image of base change, the action of $\Galois$ on the $\PGL_2(O_E)$ fixed vector is trivial.

\item
If $V$ is reducible, then $BC(\pi)$ is either Steinberg, which is ramified, or the trivial representation.  One readily checks that the trivial representations of $\PGL_2(F)$ and $\PGL_2(E) \rtimes \Galois$ match because each has character given by the constant function 1.  Therefore, $\Galois$ again acts trivially on the $\PGL_2(O_E)$-fixed vector of $BC(\pi).$
\end{itemize} 

The result follows.
\end{proof}

\subsection{Constructing a ``test function of additive reduction"}

According to Corollary \ref{nodifference}, finding a function on $\PGL_2(F)$ which trace-matches the volume 1 Haar measure on $\PGL_2(O_E)$ is equvalent to finding a test measure $m_{\text{add}}$ on $\PGL_2(F)$ satisfying

$$\tr \; \pi(m_{\text{add}}) = \begin{cases}
1 & \text{ if } \pi \text{ has additive reduction} \\
0 & \text{ otherwise}.
\end{cases}$$

The Langlands parameters of all $\pi$ of additive reduction are, in particular, tamely ramified.  By the local Langlands correspondence, they therefore have Iwahori fixed vectors.  Therefore, a natural place to start constructing such a test measure $m_{\text{add}}$ is to modify the measure $\mathbf{1}_J dj$ for congruence subgroups of Iwahori subgroups.

\subsubsection{Fixed vectors for congruence subgroups of Iwahori subgroups of $\PGL_2(F)$}
\label{fixedvectors}

Let $\mathbf{B} \subset \PGL_2 / O_F$ be a Borel subgroup and $\mathbf{N}$ be its unipotent radical and $\mathbf{T} = \mathbf{B} / \mathbf{N}.$  Let $k_F$ denote the residue field of $O_F.$  Let $C$ be any subgroup of $\mathbf{T}(k_F).$  We let $I_C \subset \PGL_2(O_F)$ denote the preimage of $C$ under the reduction map.  The subgroup $I_1$ is one such example.

\subsubsection{Supercuspidal representations have no $I_C$-fixed vectors}
\label{supercuspidal}

We quote the following theorem:

\begin{prop}[\cite{BH}, $\S 14.3$]
If $V$ is a supercuspidal representation of $\PGL_2(F),$ then $V^{I_1} = 0.$  In particular, $V^{I_C} = 0$ for any $C$ as above.  
\end{prop}

\subsubsection{$I_C$-fixed vectors for principal series representations of $\PGL_2(F)$}
\label{principalseries}

Let $B = \mathbf{B}(F), N = \mathbf{N}(F), T = \mathbf{T}(F).$  We prove a simple lemma concerning the space of $J$-fixed vectors, for an arbitrary compact open $J,$  in principal series representations $I_{\chi} = \mathrm{Ind}_B^G \chi,$ where $\chi: B \rightarrow \mathbb{C}^{\times}$ is the inflation of a character of $T.$ 

\begin{lem} \label{Jfixed}
The space of $J$-fixed vectors $(I_{\chi})^J$ is non-zero if and only if $\chi|_{B \cap gJg^{-1}} = 1$ for each $g \in B \backslash G / J.$  Furthermore, if $(I_{\chi})^J$ is non-zero, then it is $|B \backslash G / J|$-dimensional. 
\end{lem}

\begin{proof}
Let $V = I_{\chi}.$   If $f \in V^J,$ then $f(bgj) = \chi(b)f(g)$ for $b \in B, j \in J.$  In order for this to be well defined, it is necessary and sufficient that $\chi|_{B \cap gJg^{-1}} = 1$ for each $g \in B \backslash G / J.$  Furthermore, if that condition holds, an invariant function is uniquely specified by its values on $B \backslash G / J.$  Thus,  
$$\dim V^J = |B \backslash G / J| \text{ if } V^J \neq 0.$$
\end{proof}

By choosing $J = I_C$ well, we can essentially isolate all those representations of additive reduction, i.e. those whose Langlands parameters were classified in $\S \ref{tamelanglandsparameters}.$

\begin{cor}
Let $T_C = \{ t \in \mathbf{T}(O_F): t \mod \pi \in C \subset \mathbf{T}(k_F) \}$ and let $J = I_C.$  Then 
$$\dim (I_{\chi})^J = \begin{cases}
2 & \text{ if } \chi|_{T_C} = 1 \\
0 & \text{otherwise}.
\end{cases}$$ 
\end{cor}

\begin{proof}
Observe that the elements 

$$1, w = \left( \begin{array}{cc} 
0 & 1 \\
1 & 0 \\
\end{array} \right)$$

form a full set of coset representatives for $B \backslash G / J.$  But we readily compute that the image of both $B \cap wJw^{-1}$ and $B \cap J$ in $T$ is $T_C.$  The result then follows immediately by Lemma \ref{Jfixed}.
\end{proof}

\begin{cor} \label{importantsubgroup}
Let

$$C_0 = \left\{ \left(  \begin{array}{cc}
a^p & 0 \\
0 & 1
\end{array} \right): a \in k_F^{\times} \right\}.$$ 

Then 

$$\dim (I_{\chi})^{I_{C_0}} = \begin{cases}
2 & \text{ if } \chi|_{N(O_E^{\times})} = 1 \\
0 & \text{ otherwise}.
\end{cases}$$
\end{cor}

\begin{rem} 
As pointed out in $\S \ref{tamelanglandsparameters},$ the principal series representations of additive reduction are exactly those with Langlands parameters corresponding to $\chi$ for which $\chi|_{N(O_E^{\times})} = 1,$ as in the first case of the above corollary.
\end{rem}

\subsubsection{$I_{C_0}$ fixed vectors in arbitrary representations}

\begin{prop} \label{almosttestfunction}
Let $\pi$ be an irreducible admissible representation of $\PGL_2(F).$  Then 

$$\tr \; \pi(1_{C_0} dj) = \begin{cases}
0 & \text{ if } \pi = I_{\chi}, \chi|_{N(O_E^{\times})} \neq 1 \\
2 & \text{ if } \pi = I_{\chi}, \chi|_{N(O_E^{\times})} = 1 \\
1 & \text{ if } \pi = \mathbf{1} \\
1 & \text{ if } \pi = St.
\end{cases},$$

where $dj$ denotes the volume 1 Haar measure on $I_{C_0}.$  
\end{prop}

\begin{proof}
Irreducible admissible representations $\pi$ of $\PGL_2(F)$ are of three possible types:
\begin{itemize}
\item
$\pi$ is supercuspidal.  As discussed in $\S \ref{supercuspidal},$ no supercuspidal representation has an $I_{C_0}$-fixed vector.

\item
By definition, if $\pi$ is not supercuspidal then it is a subquotient of a principal series representation.  All of the principal series representations $I_{\chi},$ where $\chi \neq | \cdot |, \mathbf{1}$ are irreducible.  As explained in $\S \ref{principalseries},$ such representations have a 2-dimensional space of $I_{C_0}$-fixed vectors if $\chi|_{N(O_E^{\times})} = 1$ and a 0-dimensional space of $I_{C_0}$-fixed vectors otherwise.

\item
The representations $I_{| \cdot |}, I_{\mathbf{1}}$ are reducible.  In both cases, there is one trivial subquotient and another Steinberg subquotient. Both subquotients, the trivial representation and the Steinberg, have a 1-dimensional space of $I_{C_0}$-invariants. 
\end{itemize}

The result is simply the aggregate of all of these possibilities.
\end{proof}

\subsubsection{Isolating $\mathbf{1}$ and $St$ using the Euler-Poincar\'{e} test function}

In his work on Tamagawa numbers, Kottwitz made use of a test function which almost isolates the Steinberg representation. \bigskip

Let $C$ be a closed chamber of the building of $G = \G(F)$ for a semisimple $\G.$  For a compact open subgroup $J,$ let $dg_J$ denote the Haar measure of $G$ assigning $J$ measure 1.  Let

$$EP := \sum_{d = 0}^{\mathrm{rank } \G} \sum_{\text{stabilizers of dim } d \text{ facets of } C} (-1)^d \mathbf{1}_J dg_J.$$  

\begin{thm}[\cite{K2}] \label{EP}
For an irreducible unitary representation $\pi$ of $G,$ 

$$\tr \; \pi(EP) = \chi(H^{\bullet}(\pi) ) = \begin{cases}
1 & \text{ if } \pi = \mathbf{1} \\
(-1)^{\mathrm{rank } \G} & \text{ if } \pi = St \\
0 & \text{ otherwise}.
\end{cases} \hspace{0.5cm} (*)$$
\end{thm}

\begin{cor} \label{nonunitaryEP}
In the case of $\G = \PGL_2,$ the formula from Theorem $\ref{EP}$ holds for any irreducible admissible representation.
\end{cor}

\begin{proof}
Once again, we invoke the classification of irreducible admissible representations of $\PGL_2(F).$  Any non-unitary irreducible admissible representation must be an irreducible principal series representation is of the form $\pi_s = I_{|\cdot|^s \chi_0}.$  But $$s \mapsto \tr \; \pi_s(EP)$$ is a continuous, integer-valued function of $s$ and so must be constant and equal $\tr \; \pi_0(EP) = 0.$ 
\end{proof}

We are finally able to write down an indicator test function for those representations of additive reduction.

\begin{thm} \label{additivetestfunction}
The test function $m_{\mathrm{add}} = \frac{1}{2}[ EP + 1_{C_0} dj]$ satisfies

$$\tr \; \pi(m_{\textrm{add}}) = \begin{cases}
1 & \text{ if } \pi \text{ has additive reduction }\\
0 & \text{ if } \pi \text{ otherwise }.
\end{cases}.$$
\end{thm}

\begin{proof}
This follows immediately by combining Proposition $\ref{almosttestfunction}$ and Corollary $\ref{nonunitaryEP}.$
\end{proof}

\subsection{Definition of tame parahoric level structure and globally matching test functions} \label{redefparahoric}

We now extend the definition of parahoric level structure to include places $v$ of $F$ for which $E_v / F_v$ is tamely ramified.

\begin{defn} \label{tameparahoric}
Let $E/F$ be everywhere tamely ramified.  Also, assume that the set of places $v$ where $E_v / F_v$ is unramified is disjoint from the set of places where $D_{F_v}$ is ramified.  Let $\U_v \subset \G(E_v)$ be a compact open subgroup.  We call $\U_v$ \emph{tamely parahoric at} $v$ if

\begin{itemize}
\item
$\U_v$ is parahoric at $v$ whenever $E_v / F_v$ is split or unramified.

\item
$\U_v = \PGL_2(O_{E_v})$ if $E_v / F_v$ is tamely ramified.
\end{itemize}

We call a level structure $\U = \prod \U_v \subset \G(\mathbb{A}_E^{\mathrm{fin}})$ \emph{(globally) tamely parahoric} if $\U_v$ is tamely parahoric at $v$ for all $v.$
\end{defn}

We now consolidate all of the preceeding results of $\S \ref{tamematching}$ to provide a class of examples of globally matching test functions. 

\begin{thm} \label{finalmatching}
Let $E/F$ be everywhere tamely ramified, with $\Sigma = \{ v_1,...,v_n \}$ the places of $F$ which ramify in $E.$  Let $\U \subset \G(\finadele_E)$ be a globally tamely parahoric level structure. Let $\U^{\Sigma} = \prod_{v \notin \Sigma} \U_v$ and let $U^{\Sigma}$ denote a matching parahoric level structure outside $\Sigma.$  For $v \in \Sigma,$ let 

$$m_{\mathrm{add},v} = \frac{1}{2} \mathbf{1}_{K_v} dk_v + \frac{1}{2} \mathbf{1}_{K'_v} - \frac{1}{2} \mathbf{1}_{I_v} di_v + \frac{1}{2} \mathbf{1}_{C_0, v} dc_0 :=\sum_{i = 1}^4  c_{v,i} du_{v,i},$$

where $K_v$ and $K'_v$ are the stabilizers of two vertices in the tree for $\G(F_v) \cong \PGL_2(F_v), I_v$ is the pointwise stabilizer of the edge between them, $C_0$ is the ``preimage in $I_v$ of the $p^{th}$ powers of the diagonal torus mod $v$" (cf. Corollary \ref{importantsubgroup}), and $dk$ denotes the volume 1 Haar measure for any compact group $K.$  Then the test functions

$$\mathbf{1}_\U du \leftrightarrow \sum c_{v_1,i_1} ... c_{v_n, i_n} \mathbf{1}_{\prod_{j = 1}^n U_{v_j, i_j} \times U^{\Sigma}} \prod_{j = 1}^n du_{v_j, i_j} \times du^{\Sigma}$$

trace-match (see Definition \ref{tracematchingtestfunctions}).  
\end{thm}

\begin{proof}
Theorem \ref{additivetestfunction} shows that $m_{\mathrm{add}, v}$ trace-matches $\mathbf{1}_{\U_v} d\tilde{u}_v$ for each place $v \in \Sigma.$  For $v \notin \Sigma,$ let $U_v$ be associated to the parahoric level structure $\U_v.$  Theorem \ref{parahoricmatching} shows that $\mathbf{1}_{U_v} du_v$ trace-matches $\mathbf{1}_{\U_v} d\tilde{u}_v.$  Therefore, 
\begin{equation} \label{productmatch}
\mathbf{1}_{\U} d\tilde{u} \leftrightarrow \prod_{v \in \Sigma} m_{\mathrm{add},v} \times \mathbf{1}_{U^{\Sigma}} du^{\Sigma}
\end{equation} 
are matching test functions.  The result follows after expanding the right side of equation (\ref{productmatch}).  
\end{proof}

\section[Numerical torsion cohomology comparisons]{A numerical form of base change functoriality for torsion}
\label{funct}

In this section, we'll exhibit examples of $F$-acyclic local systems of $O_F$-modules on locally symmetric spaces associated to $\G$ for which the Cheeger-M\"{u}ller theorems of \cite[
$\S 1.7, \S 4$]{Lip1} can be applied.  In conjunction with the trace formula comparison proven in $\S \ref{abstractmatching},$ this will yield a numerical form of torsion base change functoriality.  

\begin{itemize}
\item
In $\S \ref{weirdloc},$ we will describe a certain class of $F$-acyclic local systems of $O_F$-modules over locally symmetric spaces associated to $\G.$  

\item
In $\S \ref{culmination},$ we will prove the main comparison theorem of this paper, concerning a version of ``numerical functoriality."
\end{itemize}

\subsection{A class of local systems}
\label{weirdloc}

We recall the construction of some local systems described in (\cite[$\S 10.1$]{CV}). \smallskip

Let $D$ be a quaternion algebra over an imaginary quadratic field $F$ and $\G$ be the adjoint group of the group of units of $D.$

For any quaternion algebra $Q$ over $F,$ we let $[Q]_m$ denote the $F$-vector space $[Q]_m = \mathrm{Sym}^m Q^0 / \Delta \mathrm{Sym}^{m-2} Q^0,$ where $Q^0$ denotes the trace zero elements of $Q$ and $\Delta$ denotes any invariant element of $\mathrm{Sym}^2 Q^0.$  The $F$-vector space $[Q]_m$ affords an $F$-linear representation of $Q^{\times} / F^{\times},$ which acts by conjugation on $Q$ and thus on the symmetric powers of $Q^0.$  \smallskip

Let $V_{a,b} = [D]_a \otimes [\overline{D}]_b,$ where $\overline{D}$ denotes the quaternion algebra $D \otimes_{\sigma} F, \sigma$ denoting the non-trivial automorphism of $F.$  We note that $[\overline{D}]_b = \overline{[D]_b}.$ \smallskip

Consider the representation $\rho_{a,b}: \G \rightarrow \GL(V_{a,b})$ given by conjugation action.  By the second construction outlined in $\S \ref{compatibleloc},$ there is a matching representation $\widetilde{\rho}_{a,b}: R_{E/F} \G \rtimes \Galois \rightarrow R_{E/F}\GL(V_{a,b}).$

Suppose that $U \subset \G(\finadele_F)$ is compact open and $K \subset \G(F_\R)$ is a maximal compact subgroup.  Suppose that $\U \subset \G(\finadele_E)$ is compact open and $\K \subset \G(E_\R)$ is maximal compact, with both $\U$ and $\K$ Galois-stable.  Fix an $O_F$-lattice $\mathcal{O} \subset \rho_{a,b}$ and an $O_E$-lattice $\widetilde{\mathcal{O}} \subset \widetilde{\rho}_{a,b}.$  Suppose that $U$ stabilizes $\mathcal{O}$ and $\U$ stabilizes $\widetilde{\mathcal{O}}.$  By the construction of $\S \ref{rational},$ the matching representations $\rho_{a,b}$ and $\widetilde{\rho}_{a,b}$ together with these data give rise to matching local systems
$$L_{a,b} \rightarrow M_U = \G(F) \backslash \G(\adele_F) / KU, \loc_{a,b} \rightarrow \M_\U = \G(E) \backslash \G(\adele_E) / \K \U.$$   
For each complex embedding $\iota: E \hookrightarrow \C,$ we let 
$$L_{a,b, \iota} = L_{a,b} \otimes_{\iota \circ i_F} \C, \loc_{a,b, \iota} = \loc_{a,b} \otimes_{\iota} \C,$$
local systems of $\C$-vector spaces.

\subsubsection{Acyclicity}
\label{acyclic}

We recall a theorem of Borel-Wallach, specialized to our setup:

\begin{thm}[cf. \cite{BW}, $\S$ II, Proposition 6.12] \label{acycliclocalsystems}
Let $\loc \rightarrow \M_\U$ be the local system of $\mathbb{C}$-vector spaces corresponding to a complex representation $\rho:  \G(E_\R) = R_{E/F} \G(F_\R) \rightarrow GL(V).$   Suppose that $d\rho: \mathrm{Lie}(\G(E_{\mathbb{R}}) )\rightarrow \mathrm{End}(V)$ is not isomorphic to its twist by the Cartan involution of $\G(E_\R)$ corresponding to $\K.$  Then $\loc \rightarrow \M_\U$ is acyclic.  A similar statement holds for a local system of $\C$-vector spaces $L \rightarrow M_U$ corresponding to representations of $\G(F_\R).$      
\end{thm}

\begin{proof}
See \cite[$\S 4.1$]{BV}, wherein the stronger statement that representations $\rho$ as above give rise to ``strongly acyclic families" is proven.  
\end{proof}

\begin{cor} \label{acyc}
The local systems $L_{a,b,\iota} \rightarrow M_U, \loc_{a,b, \iota} \rightarrow \M_\U$ are acyclic provided that $a \neq b.$
\end{cor}

\begin{cor}
The $O_E$-module $H^{\bullet}(\M_U, \loc_{a,b})$ and the $O_F$-module $H^{\bullet}(M_U, L_{a,b})$ are torsion for $a \neq b.$
\end{cor}

\subsection{Numerical comparison theorem}
\label{culmination}

We can finally combine our trace formula comparisons with the Cheeger-M\"{u}ller theorems of \cite[
$\S 1.7$]{Lip1} to obtain a numerical comparison of Reidemeister torsion.

\begin{thm} \label{cohomologycomparison}
Let $F$ be an imaginary quadratic field.  Let $E/F$ Galois, cyclic of odd prime degree $p.$  Let $a \neq b.$  Then there is an equality
\begin{equation} \label{numericalequality}
\log RT_{\sigma}(\M_\U, \loc_{a,b,\iota}) = p \left(\sum  c_{v_1,i_1} ... c_{v_n, i_n}  \log RT(M_{\prod_{j = 1}^n U_{v_j, i_j} \times U^{\Sigma}}, L_{a,b,\iota}) \right),
\end{equation}
where the constants $c_{v,k}$ and level structures are as in Theorem \ref{finalmatching}.
\end{thm}

\begin{proof}
By Theorem \ref{finalmatching}, the test functions
$$\mathbf{1}_\U du \leftrightarrow \sum c_{v_1,i_1} ... c_{v_n, i_n} \mathbf{1}_{\prod_{j = 1}^n U_{v_j, i_j} \times U^{\Sigma}} \prod_{j = 1}^n du_{v_j, i_j} \times du_{\Sigma}$$
are matching.  By Corollary \ref{abstracttorsion}, there is thus an equality
$$\log \tau_{\sigma}(\M_\U, \loc_{a,b,\iota}) = p \left(\sum  c_{v_1,i_1} ... c_{v_n, i_n}  \log \tau(M_{\prod_{j = 1}^n U_{v_j, i_j} \times U^{\Sigma}}, L_{a,b,\iota}) \right).$$
By the assumption $a \neq b,$ Corollary \ref{acyc} applies and so all regulators vanish.  Because $\loc_{a,b,\iota}, L_{a,b,\iota}$ are unimodular, the twisted Bismut-Zhang theorem (see \cite[
Theorem 4.1]{Lip1}) and the untwisted Cheeger-M\"{u}ller theorem of \cite{Mu2} (see \cite[
$\S 1.7$]{Lip1}) apply and give
 $$\log RT_{\sigma}(\M_\U, \loc_{a,b,\iota}) = p \left(\sum  c_{v_1,i_1} ... c_{v_n, i_n}  \log RT(M_{\prod_{j = 1}^n U_{v_j, i_j} \times U^{\Sigma}}, L_{a,b,\iota}) \right).$$
\end{proof}

\begin{rem}
The equality \eqref{numericalequality} from Theorem \ref{cohomologycomparison} is translated into a numerical comparison between torsion cohomology groups in  \cite[
Lemma 1.21]{Lip1}.
\end{rem}

\section[Growth of torsion]{Asymptotic growth of analytic torsion and torsion cohomology}
\label{growth}

Let $\G$ be the adjoint group of the unit group of a quaternion algebra $D$ over an imaginary quadratic field $F.$  Let $E/F$ be a cyclic Galois extension of odd prime degree $p.$  

In this section, we'll prove a torsion cohomology growth theorem for the local systems $\loc_{a,b}$ defined in $\S \ref{weirdloc}$ over a sequence of locally symmetric spaces for $R_{E/F} \G.$  The trace formula comparison that we have  proven in $\S \ref{abstractmatching}$ can be combined with the results of \cite{BV} to estimate the asymptotic growth of torsion in homology through a sequence $\mathcal{M}_{\mathcal{U}_N}$ of such spaces where $\U_N$ is a parahoric level structure (see definition \ref{tameparahoric}) .

\begin{itemize}
\item
We will recall the definition of a strongly acyclic family used in \cite{BV} in $\S \ref{stronglyacyclic}.$

\item
In $\S \ref{growthoftwistedatorsion},$ we will prove a growth theorem for twisted analytic torsion.

\item
In $\S \ref{cohomologygrowth},$ we will combine our growth theorem for twisted analytic torsion with the Bismut-Zhang theorem to prove a torsion cohomology growth theorem.
\end{itemize}

\subsection{Redux of strong acyclicity}
\label{stronglyacyclic}

Let $L \rightarrow M$ be a metrized local system over a compact Riemannian manifold $M.$  Let $M_n \xrightarrow{\pi_n} M$ be a sequence of covers.  \bigskip

\begin{defn}[Strong acyclicity] 
The family $L_n = \pi_n^{*} L \rightarrow M_n$ is \emph{strongly acyclic} if the spectra of the Laplace operators $\Delta_{i, L_n}, 0 \leq i \leq \dim M$ admit a uniform spectral gap.  That is, there is some $\epsilon > 0$ such that $\lambda > \epsilon$ for any eigenvalue $\lambda$ of any Laplace operator $\Delta_{i, L_n}.$ 
\end{defn}

In \cite{BV}, Bergeron and Venkatesh crucially use this hypothesis in proving a ``limit multiplicity formula for analytic torsion"; it enables them to uniformly control the long time asymptotics of an infinite family of heat kernels.    The surprising fact is that strongly acyclic local systems over locally symmetric spaces are abundant. \bigskip

Let $\mathbf{H}$ be an algebraic group over $F.$  Let $M$ be an associated locally symmetric space $\Gamma \backslash \mathbf{H}(F_{\mathbb{R}}) / K.$  Suppose that $\Gamma$ is stable by the Cartan involution $\theta$ associated to $K.$  \\
Let $\rho: \mathbf{H} \rightarrow \GL(V_{\rho})$ be a representation over $F.$  Any $O_F$-lattice inside $V_{\rho}$ which is stable by $\Gamma$ gives rise to a local system $L_{\rho} \rightarrow M.$  We will require the following strengthened version of Theorem \ref{acycliclocalsystems}.

\begin{thm}[\cite{BV}, 4.1] \label{strongbw}
Suppose $d\rho: \mathrm{Lie}(\mathbf{H}(F_{\mathbb{R}}) )\rightarrow \mathrm{End}(V) \ncong d\rho \circ \theta.$ Then for any family of covers $M_n \rightarrow M,$ the family $L_n = \pi_n^{*} L \rightarrow M_n$ is strongly acyclic.  
\end{thm}

\subsubsection{Representations of $\PGL_2(\C)$ and strong acyclicity}
\label{pgl2stronglyacyclic}

Fix a complex embedding $\iota: E \hookrightarrow \C.$

There is an isomorphism $\G(E_{\mathbb{R}}) \cong \PGL_2(\mathbb{C})^p.$  Via this identification, $\sigma$ acts by a cyclic shift.  

\begin{itemize}

\item
The irreducible representations of $\mathbf{SL}_2(\mathbb{C})$ are enumerated thus:
$$\rho_{a,b} = V_a \otimes \overline{V_b},$$
where $V_m = \mathrm{Sym}^m (\mathbb{C}^2)$ and the $\overline{\bullet}$ denotes conjugation of complex structure.  In order to factor through $\mathbf{SL}_2(\mathbb{C}) / \{ \pm 1 \} = \PGL_2(\mathbb{C}),$ we must have $a + b$ even. \bigskip

Thus, the representations of $\PGL_2(\mathbb{C})^p$ isomorphic to their $\sigma$-twists are of the form $\widetilde{\rho}_{a,b} := (V_a \otimes \overline{V_b})^{\boxtimes p}.$  The representation $\widetilde{\rho}_{a,b}$ matches the representation $\rho_{a,b} := V_a \otimes \overline{V_b}$ of $\G(F_{\mathbb{R}}) \cong \PGL_2(\mathbb{C})$  (see $\S \ref{compatibleloc}$)
\end{itemize}

\begin{rem}
$\rho_{a,b}$ from above is a slight abuse of notation.  After having fixed a complex embedding $\iota$ of $E,$ these are the homomorphisms on $F \otimes_{\mathbb{Q}} \R$-valued points induced by $\rho_{a,b}$ of $\S \ref{weirdloc}.$
\end{rem}

\begin{itemize}
\item
As explained in \cite{BV}, the representations $\rho_{a,b}$ are strongly acyclic if and only if $a \neq b.$  They verify that twisting by the (standard) Cartan involution $\theta$ sends $\rho_{a,b}^{\theta} = \rho_{b,a} \ncong \rho_{a,b}.$ By \cite[Lemma 4.1]{BV}, this implies that $\rho_{a,b}$ is strongly acyclic.

\item
By the same token, for the Cartan involution $\mathcal{\theta}$ of $\mathcal{K},$ we see that
$$\widetilde{\rho}_{a,b}^{\mathcal{\theta}} = \widetilde{\rho}_{b,a} \ncong \widetilde{\rho}_{a,b}. \text{ for } a \neq b,$$
implying that the sequence of local systems $L_{\widetilde{\rho}_{a,b,\iota}} \rightarrow \mathcal{M}_{\mathcal{U}_N}$ defined in $\S \ref{weirdloc}$ is strongly acyclic.  In particular, the rational cohomology of this entire family of local systems vanishes.  
\end{itemize}

\subsection{Growth of twisted analytic torsion}
\label{growthoftwistedatorsion}

We recall our standing notation.  Let $\G$ denote the adjoint group of the unit group of a quaternion algebra $D$ over an imaginary quadratic field $F.$  Let $E/F$ be a cyclic Galois extension of odd prime degree $p$ with Galois group $\Galois = \langle \sigma \rangle.$ 

\begin{thm} \label{refinedgrowthatorsion}
Assume that $E/F$ is everywhere tamely ramified.  Assume further that the set $\Sigma$ of places where $E/F$ is ramified is disjoint from the set of places where the quaternion algebra $D$ is ramified. Fix a complex embedding $\iota: E \hookrightarrow \C.$  \medskip

Let $\mathcal{U}_N \subset \mathcal{U}_0$ denote a sequence of compact open subgroups of $\G(\mathbb{A}_E^{\mathrm{fin}})$ such that

\begin{itemize}
\item
The injectivity radius of $\mathcal{M}_{\mathcal{U}_N}$ approaches $\infty.$

\item
The level structures $\mathcal{U}_N = \prod_v \mathcal{U}_{N,v}$ are globally \emph{tamely parahoric}.  
\end{itemize} 

Let $U_{N,v}$ be a level structure associated with $\U_{N,v}$ (see Definition \ref{associatedlevel}) for all places $v \notin \Sigma$ and let $U_{N,v}$ be an Iwahori subgroup of $\PGL_2(F_v)$ if $v \in \Sigma.$  For the matching local systems $L_{a,b}, \loc_{a,b}$ defined in $\S \ref{weirdloc},$ 
\begin{equation}
\lim_{N \rightarrow \infty} \frac{\log \tau_{\sigma}(\M_{\U_N}, \loc_{a,b,\iota})}{\vol(M_{U_N})} \rightarrow p  \cdot (1/2)^n \cdot c_{a,b} \cdot \prod_{i=1}^n \left( 2 |\mathbf{B}(k_{F_v})| - 1 + \frac{1}{q_v} \right)  \neq 0
\end{equation}
where $\mathbf{B}$ denotes the corresponding Borel subgroup of $\PGL_2 / O_{F_v}$ and $c_{a,b}$ is a non-zero constant, the $L^2$-torsion of $(\PGL_2(\C), \rho_{a,b,\iota})$ (see \cite[Theorem 4.5]{BV}).
\end{thm}

\begin{proof}
Because the level structure $\U_N$ is globally tamely parahoric, Theorem \ref{finalmatching} implies that $\mathbf{1}_{\U} du$ and $\mathbf{1}_{U^\Sigma} du^{\Sigma} \times \prod_{v \in \Sigma} m_{\mathrm{add},v}$ are matching test functions  (see Proposition \ref{almosttestfunction} and Theorem \ref{additivetestfunction} for a definition of the test function $m_{\mathrm{add},v}$).  Expand the latter measure as 
$$\mathbf{1}_{U^\Sigma} du^{\Sigma} \times \prod_{v \in \Sigma} m_{\mathrm{add},v} = \sum c_{v_1,i_1} ... c_{v_n, i_n} \mathbf{1}_{\prod_{j = 1}^n U_{v_j, i_j} \times U^{\Sigma}} \prod_{j = 1}^n du_{v_j, i_j} \times du_{\Sigma}$$
as in Theorem $\ref{finalmatching}.$  Our abstract matching theorem \ref{abstracttorsion} for analytic torsion proves that 
$$\log \tau_{\sigma}(\mathcal{M}_{\mathcal{U}_N}, \loc_{a,b, \iota} ) = p \sum c_{v_1,i_1} ... c_{v_n,i_n} \log \tau(M_{i_1,...,i_n} , L_{a,b,\iota}),$$
where
$$U_{i_1,...,i_n} := \prod_{j = 1}^n U_{v_j, i_j} \times U^{\Sigma}, M_{i_1,...,i_n} = \G(F) \backslash \G(\adele_F) / K U_{i_1,...,i_n}.$$ 
(see Theorem \ref{finalmatching} for further discussion of this notation).  For each tuple $(i_1,...,i_n),$ the family of local systems $\loc_{a,b,\iota} \rightarrow M_{i_1,...,i_n}$ is strongly acyclic, by the discussion of $\S \ref{pgl2stronglyacyclic}.$  Therefore, we may apply the ``limit multiplicity theorem for torsion" from \cite[Theorem 4.5]{BV}.  Dividing by $\vol(M_{U_N})$ and taking the limit as $N \rightarrow \infty,$ the result follows.
\end{proof}

\subsection{Cohomology growth theorem}
\label{cohomologygrowth}

\begin{thm} \label{refinedgrowthcohomology}
Enforce all the notation and assumptions of $\S \ref{growthoftwistedatorsion}.$  Let $\mathcal{U}_N \subset \mathcal{U}_0$ denote a sequence of compact open subgroups of $\G(\mathbb{A}_E^{\mathrm{fin}})$ such that

\begin{itemize}
\item
The injectivity radius of $\mathcal{M}_{\mathcal{U}_N}$ approaches $\infty.$

\item
The level structures $\mathcal{U}_N = \prod_v \mathcal{U}_{N,v}$ are globally \emph{tamely parahoric}.  

\item
The $p$-adic part of the cohomology of $\mathcal{L}_{a,b}$ is controlled as follows:   
$$\frac{\log  | H^i(\mathcal{M}_{\mathcal{U}_N}, \mathcal{L}_{a,b})[p^{\infty}]|  }{ \vol( \mathcal{M}_{\mathcal{U}_N})^{\frac{1}{p}} } \rightarrow 0.$$
%
\item
There is not too much mod $p$ cohomology in $L_{a,b},$ i.e.
$$\frac{\log |H^i(M_{U_N}, L_{a,b, \mathbb{F}_p})|}{\vol(M_{U_N}) } \rightarrow 0. $$
 \end{itemize} 

Then it follows that
$$\frac{\sum {}^{*} \log  | H^i(\mathcal{M}_{\mathcal{U}_N}, \mathcal{L}_{a,b})_{\mathrm{tors}}^{\sigma - 1}|   -   \frac{1}{p-1} \log  | H^i(\mathcal{M}_{\mathcal{U}_N}, \mathcal{L}_{a,b})_{\mathrm{tors}}^{P(\sigma)}|   }{ \vol( \mathcal{M}_{\mathcal{U}_N})^{ \frac{1}{p} }  } \rightarrow d_{a,b} \neq 0 \text{ for } a \neq b.$$
\end{thm}

\begin{proof}
Fix a complex embedding $\iota: E \hookrightarrow \C.$  By Theorem \ref{refinedgrowthatorsion}, there is a limiting identity
\begin{equation} \label{torsionlimit}
\lim_{N \rightarrow \infty} \frac{\log \tau_{\sigma}(\M_{\U_N}, \loc_{a,b,\iota})}{\vol(M_{U_N})} \rightarrow p  \cdot (1/2)^n \cdot c_{a,b} \cdot \prod_{i=1}^n \left( 2 |\mathbf{B}(k_{F_v})| - 1 + \frac{1}{q_v} \right)  \neq 0
\end{equation}
Combining \cite[
Proposition 4.5]{Lip1} with the results of \cite[
$\S 5$]{Lip1}, specifically \cite[
Example 5.6]{Lip1}, we know that
$$\tau_{\sigma}(\M, \loc_{\iota}) = RT_{\sigma}(\M, \loc_{\iota});$$ 
we have abbreviated $\M := \M_{\U_N}, \loc = \loc_{a,b}.$  Let $A^{\bullet} = \MS(X,\loc)$ denote the Morse-Smale complex for $\loc$ and a fixed weakly gradient-like vector field $X$ on $\M$ satisfying Morse-Smale transversality.  As proven in \cite[
Lemma 1.21]{Lip1},  
\begin{equation} \label{equationtobenormed}
\log RT_{\sigma}(\M, \loc_{\iota}, X) = \log | \iota(f)| - \frac{1}{p-1} \log |\iota(f')|
\end{equation}
where
$$\mathrm{Norm}_{E/\mathbb{Q}}(f) = \prod {}^{*} |H^i(A^{\bullet}[\sigma - 1])|, \mathrm{Norm}_{E/\mathbb{Q}}(f') = \prod {}^{*} |H^i(A^{\bullet}[P(\sigma)])|.$$ 
This identity is true for all embeddings $\iota.$  Summing equation (\ref{equationtobenormed}) over all the embeddings $\iota,$ we find that 
\begin{equation} \label{normedequation}
\sum_{\iota} \log \tau_{\sigma}(\M, \loc_{\iota}) = \sum {}^{*} \log |H^i(A^{\bullet}[\sigma - 1])| - \frac{1}{p-1} \sum {}^{*} \log |H^i(A^{\bullet}[P(\sigma)])|.
\end{equation}
By the estimate \cite[
$(25)_p$]{Lip1} combined with \cite[
Proposition 3.7]{Lip1} relating naive twisted Reidemeister torsion and Reidemesiter torsion, we obtain 
\begin{eqnarray*} 
\text{right side of (\ref{normedequation})} &=& - \left(\log \left|H^{*}(\M, \loc)[p^{-1}]^{\sigma - 1} \right|  -  \frac{1}{p-1} \log \left|H^{*}(\M, \loc)[p^{-1}]^{P(\sigma)} \right| \right) \\
&+& O\left( \log|H^{*}(\mathcal{M}, \loc)[p^{\infty}]| +\log |H^{*}(\mathcal{M}, \loc_{\mathbb{F}_p})| + \log |H^{*}(M, \loc_{\mathbb{F}_p})| \right),
\end{eqnarray*}
where $M$ denotes the Galois invariants of $\M.$ The remainder big $O$ term in the above equation is $o(\vol(M_{U_N}))$ by our assumption on the size of $p$-power torsion in the cohomology $H^{*}(\mathcal{M}_{\mathcal{U}_N}, \mathcal{L}_{a,b}).$   Dividing both sides of the above equation by $\vol(M_{U_N})$ and applying the limiting identity of $\eqref{torsionlimit}$ separately for each embedding $\iota$ and letting $N \rightarrow \infty$ yields the desired result.   
\end{proof}

\begin{rem} \label{calegariemerton}
The hypothesis 
$$\frac{\log |H^i(M_{U_N}, L_{a,b, \mathbb{F}_p})|}{\vol(M_{U_N}) } \rightarrow 0$$
is undesirable.  However, one expects it to hold.  For example, Calegari and Emerton in  \cite[Conjecture 1.2]{CE} have conjectured this whenever the $M_{U_N}$ vary through a $p$-adic analytic tower of hyperbolic 3-manifolds.
\end{rem}

\begin{cor} \label{unrefinedgrowth}
Assume that $E/F$ is everywhere tamely ramified of \emph{odd} prime degree $p.$  Assume further that the places where $E/F$ is ramified are disjoint from the places where the quaternion algebra $D$ is ramified. \medskip

Let $\mathcal{U}_N \subset \mathcal{U}_0$ denote a sequence of compact open subgroups of $\G(\mathbb{A}_E^{\mathrm{fin}})$ such that

\begin{itemize}
\item
The injectivity radius of $\mathcal{M}_{\mathcal{U}_N}$ approaches $\infty.$

\item
The level structures $\mathcal{U}_N = \prod_v \mathcal{U}_{N,v}$ are globally \emph{tamely parahoric}.

\item
There is not too much mod $p$ cohomology in $L_{a,b},$ i.e.
$$\frac{\log |H^i(M_{U_N}, L_{a,b, \mathbb{F}_p})|}{\vol(M_{U_N}) } \rightarrow 0. $$
\end{itemize}  

Then it follows that

$$\limsup_N \frac{\log |H^{*}(\mathcal{M}_{\mathcal{U}_N}, \mathcal{L}_{a,b} )_{\mathrm{tors}}|}{ \vol(\mathcal{M}_{\mathcal{U}_N})^{ \frac{1}{p} } } > 0 \text{ for } a \neq b.$$
\end{cor}

\begin{proof}
Suppose that the $p$-adic part of the cohomology of the sequence $(\mathcal{M}_{\mathcal{U}_N}, \mathcal{L}_{r,s})$ is large, i.e. 

$$ \frac{\log  | H^{*}(\mathcal{M}_{\mathcal{U}_N}, \mathcal{L}_{a,b})[p^{\infty}]|  }{ \vol( \mathcal{M}_{\mathcal{U}_N})^{ \frac{1}{p} }  } \nrightarrow 0 \text{ or }  \frac{\log  | H^{*}(\mathcal{M}_{\mathcal{U}_N}, \mathcal{L}_{a,b, \mathbb{F}_p })|  }{ \vol( \mathcal{M}_{\mathcal{U}_N})^{ \frac{1}{p} }  } \nrightarrow 0.$$

Then the conclusion follows vacuously. On the other hand, if there is not lots of $p$-power torsion, all hypotheses of Theorem \ref{refinedgrowthcohomology} are met and so its more refined conclusion holds ($\vol(\M_{\U_N})^{ \frac{1}{p} }$ and $\vol(M_{U_N})$ are the same order of magnitude). 
\end{proof}


\begin{thebibliography}{10}

\bibitem{AS}
Ash, Sinnott.  \emph{An analogue of Serre's conjecture for Galois representations and Hecke eigenclasses in the mod $p$ cohomology of $GL_n(\mathbb{Z})$}.   Duke Math. J. 105 (2000), no. 1, 1-24. 

\bibitem{BV}
Bergeron, Venkatesh. \emph{The asymptotic growth of torsion homology for arithmetic groups}.  Journal of the Institute of Mathematics of Jussieu, Volume 12, Issue 02 (2013), 391-447.

\bibitem{Bh}
Bhargava.  \emph{Mass formulae for extensions of local fields, and conjectures on the density of number field discriminants}.  Int. Math. Res. Not. IMRN 2007, no. 17, Art. ID rnm052, 20 pp.

\bibitem{BC}
Bismut, Cheeger.  \emph{Transgressed Euler classes of $SL_{2n}(\mathbb{Z})$ vector bundles, adiabatic limits of eta invariants and special values of $L$-functions}. Ann. Sci. École Norm. Sup. (4) 25 (1992), no. 4, 335-391. 

\bibitem{BZ}
Bismut, Zhang.  \emph{An extension of a theorem by Cheeger and M\"{u}ller. With an appendix by François Laudenbach}. Astérisque No. 205 (1992).

\bibitem{BZ2}
Bismut, Zhang. \emph{Milnor and Ray-Singer metrics on the equivariant determinant of a flat vector bundle}. Geom. Funct. Anal. 4 (1994), no. 2, 136-212.

\bibitem{BW}
Borel, Wallach. \emph{Continuous cohomology, discrete subgroups, and representations of reductive groups. Second edition}. Mathematical Surveys and Monographs, 67. American Mathematical Society, Providence, RI, 2000.

\bibitem{BH}
Bushnell, Henniart.  \emph{The local Langlands conjecture for $GL_2$}. Grundlehren der Mathematischen Wissenschaften [Fundamental Principles of Mathematical Sciences], 335. Springer-Verlag, Berlin, 2006.

\bibitem{BG}
Buzzard, Gee.  \emph{The conjectural connections between automorphic representations and Galois representations}.  Preprint.

\bibitem{CV}
Calegari, Venkatesh. \emph{A Torsion Jacquet-Langlands Correspondence}.  Preprint.

\bibitem{CE}
Calegari, Emerton.  \emph{Mod-$p$ cohomology growth in $p$-adic analytic towers of 3-manifolds}. Groups Geom. Dyn. 5 (2011), no. 2, 355-366.

\bibitem{Ch}
Cheeger.  \emph{Analytic torsion and the heat equation}. Ann. of Math. (2) 109 (1979), no. 2, 259-322.

\bibitem{GR}
Gross, Reeder.  \emph{Arithmetic invariants of discrete Langlands parameters}. Duke Math. J. 154 (2010), no. 3, 431-508.

\bibitem{Ill}
Illman.  \emph{Smooth equivariant triangulations of $G$-manifolds for $G$ a finite group}.
Math. Ann. 233 (1978), no. 3, 199-220.

\bibitem{JL}
Jacquet, Langlands. \emph{ Automorphic forms on $GL_2$}. Lecture Notes in Mathematics, Vol. 114. Springer-Verlag, Berlin-New York, 1970.

\bibitem{KM}
Knudsen, Mumford.  \emph{The projectivity of the moduli space of stable curves. I. Preliminaries on ``det'' and ``Div''}. Math. Scand. 39 (1976), no. 1, 19-55. 

\bibitem{K}
Kottwitz. \emph{Base change for unit elements of Hecke algebras}.   Compositio Math. 60 (1986), no. 2, 237-250. 

\bibitem{K1}
Kottwitz.  \emph{Rational conjugacy classes in reductive groups}.  Duke Math. J. 49 (1982), no. 4, 785-806.

\bibitem{K2}
Kottwitz.  \emph{Stable trace formula: the elliptic singular terms}.  Math. Ann. 275, Issue 3 (1986), 365-399. 

\bibitem{Lan}
Langlands.  \emph{Base change for $GL_2$}. Annals of Mathematics Studies, 96. Princeton University Press, Princeton, N.J.; University of Tokyo Press, Tokyo, 1980.

\bibitem{Lip1}
Lipnowski.  \emph{The Equivariant Cheeger-M\"{u}ller Theorem on Locally Symmetric Spaces}.  Preprint.

\bibitem{LR}
Lott, Rothenberg. \emph{Analytic torsion for group actions}.   J. Differential Geom. 34 (1991), no. 2, 431-481.

\bibitem{Lu}
L\"{u}ck.  \emph{Analytic and topological torsion for manifolds with boundary and symmetry}. J. Differential Geom. 37 (1993), no. 2, 263-322.

\bibitem{Mu1}
M\"{u}ller.  \emph{Analytic torsion and R-torsion of Riemannian manifolds}.  Adv. in Math. 28 (1978), no. 3, 233-305.

\bibitem{Mu2}
M\"{u}ller.  \emph{Analytic torsion and R-torsion for unimodular representations}.   J. Amer. Math. Soc. 6 (1993), no. 3, 721-753.

\bibitem{O}
Oesterl\'e.  \emph{Nombres de Tamagawa et groupes unipotents en caract\'eristique $p$}.   Invent. Math. 78 (1984), no. 1, 13-88.

\bibitem{RS}
Ray, Singer. \emph{R-torsion and the Laplacian on Riemannian manifolds}. Adv. in Math. 7 (1971), 145-210. 

\bibitem{Sen}
Seng\"{u}n.  \emph{On the integral cohomology of Bianchi groups}.   Exp. Math. 20 (2011), no. 4, 487-505.

\bibitem{Sch}
Scholze.  \emph{On torsion in the cohomology of locally symmetric varieties}.  Preprint.

\bibitem{Se}
Serre.  \emph{Sur les repr\'esentations modulaires de degr\'e 2 de $\mathrm{Gal}(\overline{\mathbb{Q}}/ \mathbb{Q})$}.   Duke Math. J. 54 (1987), no. 1, 179-230. 

\end{thebibliography}

\end{document}